\documentclass[11pt,reqno]{amsart}

\usepackage{comment}
\usepackage[utf8]{inputenc}
\usepackage[english]{babel}
\usepackage{amsmath}
\usepackage{mathtools}
\mathtoolsset{showonlyrefs}
\usepackage{amsfonts,amssymb}
\usepackage{mathrsfs,enumerate,doi}
\usepackage{booktabs}
\usepackage{esint} 
\usepackage{amsthm}
\usepackage{nicefrac}
\usepackage{hyperref}
\usepackage{accents}
\usepackage{epsfig}
\usepackage{graphicx}
\usepackage{csquotes}
\usepackage{psfrag}
\usepackage{graphpap,latexsym,epsf}
\usepackage{xcolor}
\usepackage{geometry}
\newtheorem{theorem}{Theorem}[section]
\newtheorem{corollary}[theorem]{Corollary}
\newtheorem{lemma}[theorem]{Lemma}
\newtheorem{prop}[theorem]{Proposition}
\newtheorem{definition}[theorem]{Definition}

\newtheorem{obs}[theorem]{Remark}

\newtheorem{ackn}{Acknowledgement\ \!\!\!\!\!}

\newcommand{\blu}[1]{{\color{blue}#1}}

\newcommand{\Om}{\Omega}
\newcommand{\R}{\mathbb{R}}
\newcommand{\D}{\mathrm{D}}
\newcommand{\f}{\varphi}
\newcommand{\N}{\nabla}
\newcommand{\e}{\varepsilon}
\newcommand{\pa}{\partial}
\hypersetup{colorlinks=true, linkcolor=blue, filecolor=black, urlcolor=blue,
citecolor=teal}
\usepackage{xcolor}
\usepackage{hyperref}
\definecolor{vgreen}{rgb}{0.1,0.5,0.2}

\title{$p$-Laplace Serrin problem on planar annuli}

\author[V. Agostiniani]{Virginia Agostiniani}
\author[C. Bernardini]{Chiara Bernardini}

\address[V. Agostiniani]{
Department of Mathematics,
University of Trento 
\newline\indent Via Sommarive, 14 - 38123 Povo (Trento)}
\email{\href{mailto:virginia.agostiniani@unitn.it}{virginia.agostiniani@unitn.it}}

\address[C. Bernardini]{
Department of Mathematics,
University of Rome Tor Vergata 
\newline\indent Via della Ricerca Scientifica 1, 00133 Roma}
\email{\href{mailto:bernardini@axp.mat.uniroma2.it}{bernardini@axp.mat.uniroma2.it}}

\begin{document}

\begin{abstract}
We prove a rigidity result for the problem $\Delta_p u=-2$ defined on a ring-shaped domain in $\R^2$, when homogeneous Dirichlet boundary conditions are imposed and $p\in(1,2)$. 
This problem admits a natural physical interpretation, as it describes the steady laminar flow of a generalised Newtonian fluid exhibiting shear-thinning behaviour.
Our symmetry result is obtained assuming that the set of maximum points, denoted by $\Sigma$, is a smooth closed curve, and there are no other critical points outside it; moreover, an appropriate asymptotic expansion of the solution near $\Sigma$ is required. 
Although our approach is inspired by the comparison strategy of \cite{AgoBerBorMaz}, the main ingredients used there are intrinsically linear and have no direct counterparts for the \(p\)-Laplacian. This requires the introduction of new ideas and techniques.
\end{abstract}

\maketitle
\date{}

{
\hypersetup{linkcolor=black}
\tableofcontents
}

%----------------------------------------------------------------

\section{Introduction}

In this paper, we establish symmetry results for solutions to the $p$-Laplacian torsion problem on planar domains. For a fixed $p\in(1,2)$, 
we consider pairs $(\Omega,u)$ satisfying the boundary value problem
\begin{equation}\label{nostro problema}
\begin{cases}
    \Delta_p u=-2, \quad&\text{in }\Omega,\\
    u=0,  & \text{on }\partial\Omega.\
\end{cases}
\end{equation}

We are mainly concerned with the case 
where $\Omega$ is a planar \textit{ring-shaped domain}, that is a bounded domain in $\mathbb{R}^2$ with a smooth boundary consisting of exactly two connected components:
\[
\partial\Omega=\Sigma_{\rm i}\sqcup\Sigma_{\rm o},
\]
where \(\Sigma_{\rm i}\) denotes the inner boundary component and \(\Sigma_{\rm o}\) the outer one.
We recall that problem \eqref{nostro problema} admits an interesting physical interpretation in the context of steady laminar flows of non-Newtonian fluids.
More precisely, starting from the Cauchy equation of motion and assuming an incompressible, steady, laminar flow in a pipe of cross section $\Omega$, the model for a generalized Newtonian fluid with shear-dependent viscosity reduces to the equation
\[
\operatorname{div}\big(\mu(|\nabla u|)\nabla u\big)=c,
\]
where $u$ denotes the vertical component of the velocity field. Choosing $\mu(s)=Ks^{p-2}$, which corresponds to the $p$-Laplacian, the cases $p>2$, $1<p<2$, and $p=2$ describe, respectively, shear-thickening, shear-thinning, and Newtonian fluid behaviour. We refer the reader to Section \ref{sez2} for more details on this derivation.\\

Starting from the celebrated rigidity result for Serrin’s problem in \cite{Serrin}, several works have been devoted to extending his result to more general operators than the classical Laplacian. As a consequence, overdetermined problems for quasilinear operators have been intensively studied over the past few decades. Philippin \cite{Philippin} studied the following quasilinear (nondegenerate) equation
$${\rm div}\big(f(|\nabla u|)\,\nabla u)=-2\qquad\text{in }\Omega$$
imposing
$$u=0, \quad \text{ and }\quad|\nabla u|=c\qquad\text{on }\partial\Omega$$
where $\Omega\subset\R^n$ is a connected bounded domain and $f$ is bounded above and below by positive constants. He proved a rigidity result using a technique first introduced by Weinberger in \cite{Wein}, and exploiting a maximum principle argument together with Rellich’s identity. 
Later Garofalo and Lewis \cite{GarofaloLewis} established the symmetry result for a broader class of (potentially) degenerate equations, under growth conditions on $f$, and imposing `weaker' Neumann boundary conditions which do not require smoothness assumptions on the boundary of the domain. Their result, which includes the $p$-Laplacian as a special case, is obtained by generalizing Weinberger's argument. By contrast, Brock and Henrot \cite{BroHen} adopted a completely different strategy, relying on Steiner symmetrization and domain derivative techniques to investigate an overdetermined problem for a class of quasilinear operators of $p$-Laplacian type defined on convex domains. 
We also refer the interested reader to the work by Fragalà, Gazzola and Kawohl \cite{FrGaKo} (see also Farina and Kawohl \cite{FarinaKaw}) who studied a family of degenerate equations not covered by the above-mentioned papers, and to \cite{CiraoloRoncoroni} for problems on convex cones considering possibly degenerate operators in Euclidean space or space forms.
Concerning symmetry results for domains with cavities, they have already been considered by Alessandrini \cite{Alessandrini} and Payne-Philippin \cite{PhiPay89}, in both cases, they studied a degenerate quasi-linear capacity overdetermined problem, and in particular, the equation treated is 
$${\rm div}\big(f(u,|\nabla u|)\,\nabla u)=0\qquad\text{in }\,\Omega_0\setminus\cup_{k=1}^m\bar{\Om}_k.$$
Along these lines, the non-homogeneous case has been treated by Payne and Philippin in \cite{PaynePhil}; they assume $f$ satisfies a uniform ellipticity condition and impose more restrictive conditions (than ours) on the boundary data.
Very recently, Cao, Wei, and Zhan \cite{cao} considered overdetermined problems involving possibly degenerate elliptic equations of the form 
$$-{\rm div}\left(f(|\nabla u|)\frac{\nabla u}{|\nabla u|}\right)=g(u)$$
on ring-shaped domains $\Om=\Om_0\setminus\overline\Om_1$, under suitable assumptions on $f$ and $g$. They impose a weaker version of the previous overdetermining conditions, which, however, in the case $u\in C^1(\bar\Om)$, reduces to $u=0$, $\partial_\nu u=c_0$ in $\partial\Om_0$, and $u=\eta$, $\partial_\nu u=c_1$ in $\partial\Om_1$. 
In particular, assuming $\eta>0$, $c_0\le0$ and $c_1\ge0$, they prove that, if $u\in C^1(\Om)$ is a weak solution to the problem, then $\Om_0$ and $\Om_1$, must be (possibly non-concentric) balls.
Notice that, in general, their solution $u$ may not be radially symmetric. Their proof relies on continuous Steiner symmetrization together with a suitable approximation argument.
We also refer the reader to the recent work by Figalli and Zhang \cite{fiZh} where they address an overdetermined problem for the anisotropic Laplacian $\Delta_H u ={\rm div}(H(\nabla u) {\rm D}H(\nabla u))$ on a bounded indecomposable set of finite perimeter. Their result provides an anisotropic Serrin-type rigidity theorem for Lipschitz domains and, more generally, for finite-perimeter sets satisfying suitable rectifiability assumptions.
We finally recall the work by Reichel \cite{Reichel} on analogous quasilinear problems in exterior domains, as well as \cite{AleRos, EncPer, Poggesi} for what concerns the $p$-harmonic case on punctured domains. \\

In what follows, we always assume $\Om$ a planar bounded domain with smooth boundary, and $p\in(1,2)$, unless explicitly stated otherwise. Notice that by maximum principle arguments, if $u$ solves \eqref{nostro problema} then $u>0$ in $\Omega$ and the function $u$ achieves its maximum in the interior of the domain. We will use the following notation:
\[
u_{\mathrm{max}}:=\max_{\Omega}{u},\qquad\quad{\rm MAX}(u)=\{x\in\Omega\,|\,u(x)=u_{\mathrm{max}}\}.
\]

We will always assume the following.

\newtheorem*{ip no crit}{Assumption I}
\begin{ip no crit}
The solution $u$ has no critical points in the interior of $\Omega$ other than its maximum points; namely,
\[
{\rm Crit}(u)\setminus {\rm MAX}(u)=\varnothing.
\]
\end{ip no crit}

We stress that Assumption I should be regarded essentially as a working assumption. We expect that it can be removed by means of a suitable \(\varepsilon\)-regularization of the solution, in the sense of DiBenedetto \cite{DiBen}. For the present purposes, however, we retain this assumption in order to develop the theory in a transparent way and to isolate the main ideas in this basic setting. A complete treatment allowing for the presence of additional critical points is deferred to future work.
It is therefore important to keep in mind that Assumption I is of a different nature from the more structural hypotheses introduced next, which enter the statements of our main results.

\begin{definition}\label{ass_Loja}
We say that $u$ satisfies a
\emph{$p$-{\L}ojasiewicz inequality} at ${\rm MAX}(u)$, if there exist a neighbourhood $\mathscr U$ of ${\rm MAX}(u)$ and two constants  $c>0$ and 
$\frac 1p\leq\gamma<\frac2p$ such that
\begin{equation}\label{loj con p}
|\nabla u(x)|\,\geq \,c\,(u_{\rm max}-u(x))^\gamma,
\qquad\mbox{for every }\ x\in\mathscr U.    
\end{equation}
\end{definition}

Observe that when $p=2$ any solution $u$ of \eqref{nostro problema} satisfies \eqref{loj con p}. Indeed, in this case $u$ is real analytic and \eqref{loj con p} reduces to the classical {\L}ojasiewicz Inequality.
Observe as well that, in the range $1<p<2$ considered here, condition \eqref{loj con p} is less restrictive than the classical {\L}ojasiewicz inequality, since it allows for larger exponents.

\begin{obs}\label{scal_inv}
{\rm
If $0\in\Om$ and $\lambda>0$, then $u$ solves \eqref{nostro problema} in $\Om$ if and only if
$\lambda^{\frac p{p-1}} u(x/\lambda)$ solves \eqref{nostro problema} in $\lambda\Om$.
In view of this fact, we define the following model solutions:
$$\mathfrak u_0(x)=\,\Big(\frac{p-1}p\Big)\, \big(1-|x|^{\frac p{p-1}}\big), \quad x\in B_1(0),$$
$$\mathfrak u_0^\lambda(x)=\lambda^{\frac p{p-1}}\mathfrak u_0\left(\frac x{\lambda}\right),\quad x\in B_\lambda(0).$$
In what follows, we will refer to solutions of this form as \textit{p-Serrin solutions}.
Note that the model solution $\mathfrak u_0$ fulfils \eqref{loj con p} with $\vartheta=\frac1p$.
}
\end{obs}

First, we obtain a characterization of $p$-Serrin solutions in terms of the quantity $\max_{\partial\Om}\frac{|\nabla u|^p}{u_{\rm max}}$. 

\begin{theorem}\label{teo2.1NEW}
Let $(\Omega,u)$ be a solution to problem \eqref{nostro problema}. 
Under Assumption {\rm \blu{I}}, suppose that $u$ satisfies a 
$p$-{\L}ojasiewicz inequality at ${\rm MAX}(u)$ and also a $\frac{2-p}{p^2}$-reverse {\L}ojasiewicz inequality at at ${\rm MAX}(u)$, according to Definition \ref{RRLI}.
If $N$ is a connected component of $\Omega\setminus{\rm MAX}(u)$ and 
$$\max_{\partial\Om\cap\overline N}\frac{|\nabla u|^p}{u_{\rm max}}\le\frac{p}{p-1},$$
then $\Om$ is a disk and $u$ coincides, up to translations, with $u_0^\lambda$, for some $\lambda>0$.
\end{theorem}

\noindent 
Observe that the quantity $|\nabla u|^p/u_{\rm max}$ appearing in the above statement is a scaling-invariant quantity, in view of Remark \ref{scal_inv}.
We also emphasize that the two {\L}ojasiewicz-type inequalities appearing in the statement play distinct roles in the proof. The {\L}ojasiewicz inequality is used to show that $\Omega\setminus \operatorname{MAX}(u)$ is connected (see Proposition \ref{prop2.3new}), whereas the reverse {\L}ojasiewicz inequality is needed to obtain a suitable gradient estimate (see Theorem \ref{teo p serrin omega}). For functions that are $C^2$ across the set of maximum points, such a reverse inequality follows from \cite[Theorem 2.2]{BoChMa}. 
In the $C^{1,\alpha}$ setting considered here, we expect an analogous inequality to hold with a suitable exponent 
$\bar\gamma<1$ (see Definition \ref{RRLI}). This would allow us to remove the reverse {\L}ojasiewicz inequality from the assumptions in the above statement.

Theorem \ref{teo2.1NEW}, which we believe to be of independent interest, will be used in the proof of our second main result, which characterizes annular-type rotationally symmetric solutions to \eqref{nostro problema} with locally constant Neumann data.
These solutions are completely described, up to translations and rescalings, by the following family
$\{\mathfrak u_R\}_{0<R<1}$.
For every fixed $0<R<1$, letting 
$$(\mathfrak u_R)_{\rm max}:=\int\limits_{R}^{1}\left(\rho-\frac{R^2}{\rho}\right)^\frac{1}{p-1}d\rho,$$ 
it can be easily proved that 
there exists a unique $r_i\in(0,R)$ such that
$$\int\limits_{r_i}^{R}\left(\frac{R^2}{\rho}-\rho\right)^\frac{1}{p-1}d\rho=(\mathfrak u_R)_{\rm max}.$$
Then the function
\begin{equation}\label{model_annular_slus}
\mathfrak u_R(x)=\begin{cases}
\displaystyle\int_{r_i}^{|x|}\left(\frac{R^2}{\rho}-\rho\right)^\frac{1}{p-1}d\rho\qquad\quad\qquad &\text{if }r_i<|x|\le R\\
\displaystyle
(\mathfrak u_R)_{\rm max} -\int_{R}^{|x|}\left(\rho-\frac{R^2}{\rho}\right)^\frac{1}{p-1}d\rho &\text{if }R\le|x|<1
\end{cases}
\end{equation}
solves problem \eqref{nostro problema}
in the circular annular domain
\begin{equation}\label{model annular domain}
\Om_R=\left\{x\in\R^2\,\big|\,r_i<|x|<1\right\}.   
\end{equation}
Figure~\ref{fig:annular-profiles} illustrates the profiles of the model annular solutions $\mathfrak u_R$ for some values of $p\in(1,2)$, when \(R=\frac12\). 

\begin{figure}[htbp]
\centering
\includegraphics[width=0.70\textwidth]{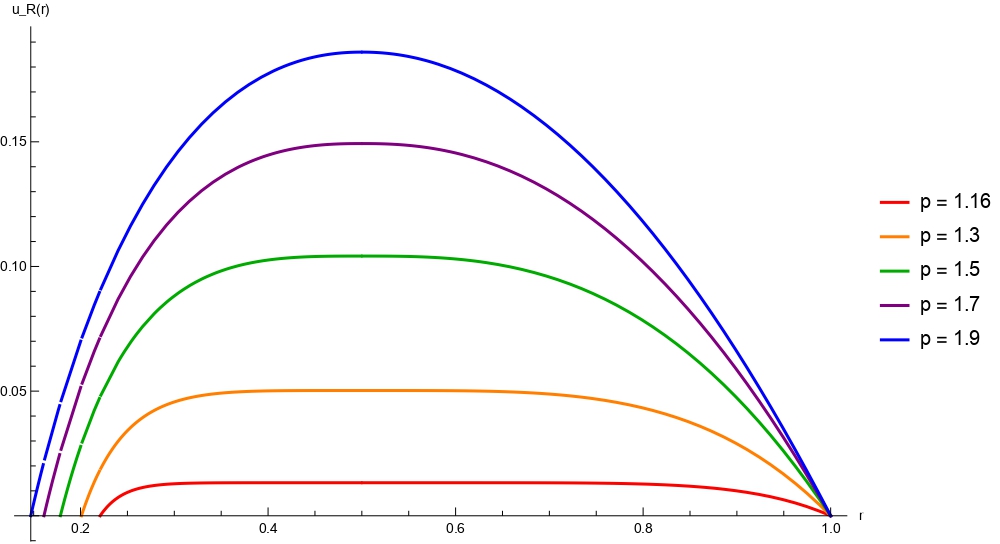}
\caption{\small{Profiles of model annular solutions $\mathfrak u_R$ for some values of $p\in(1,2)$, with $R=\frac12$.}}
\label{fig:annular-profiles}
\end{figure}

\noindent As we will see in Sections \ref{sez stime grad} and \ref{sez6}, a key step in the proof of our results is to establish a suitable gradient estimate, namely, to control the gradient of the solution in terms of an appropriate model comparison solution. The profiles in Figure \ref{fig:annular-profiles} provide some intuition for why this issue becomes increasingly delicate as $p$ approaches $1$: the model solutions develop an increasingly flat region around their maximum.

Analysing the first derivative of $\mathfrak u_R$, it is easy to see that
$\mathfrak u_R\in C^2$ for every $p\in(1,2)$.
Moreover, for every natural number $k\geq3$, we have that $\mathfrak u_R\in C^k$ if $p\in(1,\nicefrac{k}{k-1})$ and $u_R\notin C^k$ if $p\in(\nicefrac{k}{k-1},2)$. At the same time, $u_R\in C^{k-1}\setminus C^k$ if $p=\nicefrac{k}{k-1}$ and $k$ is odd, whereas $u_R\in C^{\infty}$ if $p=\nicefrac{k}{k-1}$ and $k$ is even. That is, for every natural number $k\geq3$, 
\[
u_R\in C^k\mbox{\qquad if and only if\qquad} 1<p<\frac{k}{k-1}\mbox{\quad or\quad}p=\frac{m}{m-1}\mbox{\ for some even $m$}.
\]

Observe that the $p$-Serrin solution arises as the singular limit of the ring-shaped model solutions in the regime $R \to 0^+$. More precisely, the family \(\mathfrak u_R\) converges smoothly to \(\mathfrak u_0\) on every compact subset of
\(B_1(0)\setminus\{0\}\): see Figure~\ref{fig:profili_R_var} for a comparison of the profiles corresponding to different values of $R$.
In particular, the sets ${\rm MAX}(\mathfrak u_R)$ and $\{|x|=r_i\}$ both collapse onto the origin as \(R \to 0^+\), moreover $(\mathfrak u_R)_{\rm max}\to\frac{p-1}{p}$ as $R\to0^+$.

\begin{figure}[htbp]
\centering
\includegraphics[width=0.7\textwidth]{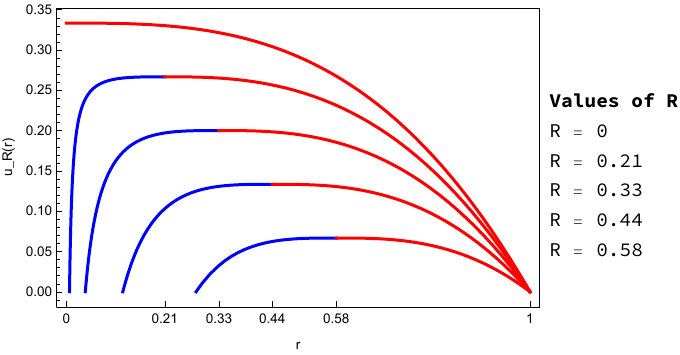}
\caption{\small{Radial profiles of the ring-shaped model solutions 
$\mathfrak u_R$ for \(p=\frac32\) and different values of \(R\). The blue and red portions represent, respectively, the increasing and decreasing branches of each profile. The case \(R=0\), corresponding to the \(p\)-Serrin solution \(\mathfrak u_0\), is included for comparison.}}
\label{fig:profili_R_var}
\end{figure}

Our main rigidity result is the following.

\begin{theorem}\label{TEOREMA rigidita}
Let $\Omega$ be a planar ring-shaped domain, and let $u$ be the solution to problem \eqref{nostro problema} such that 
\[|\nabla u| \quad\mbox{is locally constant on}\quad\partial\Om. \]
Under Assumption {\rm \blu{I}}, assume further that $\Sigma:={\rm MAX}(u)$ is a smooth closed (nontrivial) curve and that the following asymptotic expansion holds for $u$ at $\Sigma$:
\begin{equation}
\label{ass_esp}
u=u_{\rm max}+a(\theta)\,t^\frac{p}{p-1}+b(\theta)\,t^{\frac{p}{p-1}+1}\,+\,o_1 \left(t^{\frac{p}{p-1}+1}\right),
\end{equation}
where $t$ is the distance function from $\Sigma$.

Then, $\Om$ is a circular annular domain and the solution $u$ is rotationally symmetric.
\end{theorem}

\vspace{2mm}
\noindent In the above statement, by $f=o_1(t^\ell)$ we mean that $\frac{\partial f}{\partial\theta}=o(t^\ell)$, $\frac{\partial f}{\partial t} =o(t^{\ell-1})$, and $\frac{\partial^2 f}{\partial t \partial \theta} =o(t^{\ell-1})$.

\begin{obs}\label{II implica no bnd su p}
\rm{
The proof of Theorem \ref{TEOREMA rigidita} takes advantage of Theorem \ref{teo2.1NEW}, 
whose assumptions are satisfied. 
Indeed, expansion \eqref{ass_esp} implies, in particular, the $p$-{\L}ojasiewicz inequality of Definition \ref{ass_Loja}. This is because 
letting $q:=\frac{p}{p-1}$, since Proposition \ref{lemma7.1} guarantees the non-degeneracy of the leading coefficient \(a\) and in particular that $a(\theta)\leq -c_0<0$ on $\Sigma$, then $(u_{\rm max}-u)\sim t^q$ and 
$|u_t|\sim t^{q-1}$,
as \(t\to0\), uniformly with respect to \(\theta\). This implies that \(u\) satisfies inequality \eqref{loj con p} with $\gamma=\frac1p$. The asymptotic expansion \eqref{ass_esp} also implies that \(u\) satisfies a $\frac1p$-reverse {\L}ojasiewicz inequality and, consequently, also a $\frac{2-p}{p^2}$-reverse {\L}ojasiewicz inequality (see Definition \ref{RRLI}).
}
\end{obs}

\medskip

The main tool to prove our symmetry result is the \textit{comparison geometry technique} introduced in \cite{AgoBorMaz}, where it was employed to characterize rotationally symmetric solutions to the classical torsion problem defined over planar annular domains (see also \cite{AgoBerBorMaz} for the $n$-dimensional case). The same technique also found fruitful applications in the context of general relativity \cite{BoChMa}, in the study of the critical catenoid conjecture \cite{EspMar} and semilinear equations \cite{esFeMa}. To the best of our knowledge, this is the first attempt to extend such a procedure to equations involving quasilinear operators.

In order to start the comparison procedure, we split the problem into two sub-problems as suggested by the following observation. Under the assumptions of Theorem \ref{TEOREMA rigidita}, the curve $\Sigma$ divides the domain $\Omega$ into two ring-shaped domains, more specifically
$$\Omega=\Omega_{\rm i}\sqcup\Sigma\sqcup\Om_{\rm o},$$ 
where $\Omega_{\rm i}$ denotes the inner region of $\Om$ having boundary $\partial\Om_{\rm i}=\Sigma\sqcup\Sigma_{\rm i}$, and $\Om_{\rm o}$ the outer one such that $\partial\Om_{\rm o}=\Sigma\sqcup\Sigma_{\rm o}$ (to keep the presentation concise, we postpone the details about the curve $\Sigma$ to \textit{Remark} \ref{no isl}). 
Moreover, since $|\nabla u|$ is locally constant on $\partial\Om$, without loss of generality, we can assume that $\partial_\nu u=\alpha$ on $\Sigma_{\rm i}$ and $\partial_\nu u=\beta$ on $\Sigma_{\rm o}$, where $\alpha$ and $\beta$ real constants.
As a consequence, the key idea in our argument consists of studying separately the rigidity of the solution in each of these ring-shaped sub-domains.  More precisely, the problem reduces to the study of the following two overdetermined sub-problems
\begin{equation}\label{sist interno}
\begin{cases}
\Delta_pu=-2,\qquad\qquad&\text{in }\Omega_{\rm i}\\
u=0,\quad\hspace{0,29cm}\frac{\partial u}{\partial\nu}=\alpha,&\text{on }\Sigma_{\rm i}\\
u=M,\quad\,\frac{\partial u}{\partial\nu}=0,&\text{on }\Sigma\\
0\,\le\,u\,\le\,M, &\text{in }\Omega_{\rm i}
\end{cases} 
\end{equation}
and
\begin{equation}\label{sist esterno}
\begin{cases}
\Delta_p u=-2,\qquad\qquad&\text{in }\Omega_{\rm o}\\
u=M,\,\,\quad\frac{\partial u}{\partial\nu}=0,&\text{on }\Sigma\\
u=0,\quad\,\hspace{0.28cm}\frac{\partial u}{\partial\nu}=\beta,&\text{on }\Sigma_{\rm o}\\
0\,\le\,u\,\le\,M, &\text{in }\Omega_{\rm o}
\end{cases}    
\end{equation}
where $\nu$ denotes the exterior unit normal to $\Omega_{\rm i}$ of $\Om_{\rm o}$, respectively. Since we are under Assumption {\rm \blu{I}}, throughout this paper, whenever we consider either system \eqref{sist interno} or \eqref{sist esterno}, we always assume that the solution has no critical points in the interior of the domain $\Om_{\rm i}$ or $\Om_{\rm o}$, respectively.  Moreover, we will always assume that $M>0$. This comes from the observation that if $u$ solves \eqref{nostro problema} then $u>0$ in $\Omega$ and the function $u$ achieves its maximum in the interior of the domain, so we are interested in the case when $M=u_{\rm max}>0$.\\

Therefore, given $a,b,\alpha,\beta$ real constants, we are interested in studying symmetry for the pairs $(\Om,u)$, assuming that $\Omega\subset\R^2$ is a ring-shaped domain having smooth boundary, and $u:\Om\to\R$ is a solution to the following overdetermined problem 

\begin{equation}\label{sist a b alpha beta}
\begin{cases}
\Delta_pu=-2,\qquad\qquad&\text{in }\Omega\\
u=a,\quad\frac{\partial u}{\partial\nu}=\alpha,&\text{on }\Gamma_\mathrm{i}\\
u=b,\quad\,\frac{\partial u}{\partial\nu}=\beta,&\text{on }\Gamma_\mathrm{o}.
\end{cases}  
\vspace{1mm}
\end{equation}
where $\nu$ is the exterior unit normal to $\Omega$, $\Gamma_{\rm i}$ is the inner boundary component of $\Om$ and $\Gamma_{\rm o}$ the outer one.
By considering the cases $ a < b$ and $ a> b$ separately, we obtain the following symmetry result, which plays a fundamental role in the proof of Theorem \ref{TEOREMA rigidita}. 
Once again, when considering problem \eqref{sist a b alpha beta}, we always assume that the desired solution does not have critical points inside the domain $\Omega$.

\begin{theorem}\label{teo rigidità a<b e a>b}
Let $(\Omega,u)$ be a solution to problem \eqref{sist a b alpha beta} such that ${\rm Crit}(u)\cap\Om=\varnothing$. Assume either:
\begin{itemize}
    \item [i.] $a<u<b$ in $\Omega$ and $\beta\not=0$;
\end{itemize}
or
\begin{itemize}
    \item [ii.] $b<u<a$ in $\Omega$, $\alpha\not=0$ and $a-b\le\left(\frac{p-1}{p}\right)(|\beta|^p-\alpha^p)$.
\end{itemize}
Then, $\Omega$ is a circular annular domain, and the solution $u$ is rotationally symmetric.
\end{theorem}

\noindent The cases where $\beta=0$ or $\alpha=0$ are addressed in Theorem \ref{teo rigidità caso beta nullo}. 
The case $p=2$ was treated in \cite{borghini}, and the present work extends both the technique developed there and the corresponding result to the range $p\in(1,2)$. We also point out that it should be possible to obtain similar results by suitably adapting Sirakov's work \cite{Sirakov}, which investigates a similar problem for a broader class of elliptic equations, including the $\varepsilon$-regularised $p$-Laplacian.

We emphasize that Theorem \ref{teo rigidità a<b e a>b} addresses the case of nonvanishing Neumann data and is therefore not sufficient, by itself, to prove Theorem \ref{TEOREMA rigidita}, for which the degenerate cases $\beta=0$ or $\alpha=0$ also need to be considered. These cases are covered by Theorem \ref{teo rigidità caso beta nullo}. We have nevertheless chosen to state Theorem \ref{teo rigidità a<b e a>b} in this form, without including the case of vanishing Neumann data, in order to highlight that, once such data are excluded, the symmetry result can be obtained without assuming any extra asymptotic expansion near the maximum points, in contrast with Theorem \ref{teo rigidità caso beta nullo}.\\

Our main goal is to compare the gradient of $u$ with the gradient of a suitably chosen ring-shaped model solution $(\Om_R,u_R)$, to finally obtain a sharp and rigid upper bound for the former in terms of the latter (see Theorems \ref{stime gradiente} and \ref{stime gradiente esterne}). 
First, we need to associate a solution to problem \eqref{sist a b alpha beta} with an appropriate ring-shaped model solution to compare with. Therefore, we show that for any choice of parameters $M$ and $\alpha$, satisfying suitable necessary conditions (see Corollaries \ref{cor cond poho} and \ref{lemma a<b caso beta=0}), there exists a pair $(\Om_R, u_R)$, such that
\begin{equation*}
u_R(x)=\int_{r_i}^{|x|}\left(\frac{R^2}{\rho}-\rho\right) ^\frac{1}{p-1}d\rho,
\end{equation*}
over the domain
\begin{equation*}
\Omega_R=\{x\in\R^2\,|\,r_i<|x|<R\},
\end{equation*}
solves system \eqref{sist interno} and acts as model inner solution. A similar construction can be done for problem \eqref{sist esterno}.  
Once we have constructed the model, we introduce the so-called \textit{pseudo-radial function} $\Psi$, that is, a function that mimics the behaviour of the radial coordinate, and coincides with $|x|$ if and only if the solution $(\Om,u)$ coincides with the model one. In turn, we can rephrase our problem in terms of the conformal metric $g=\frac{1}{\Psi^2}g_{\R^2}$. 
This equivalent reformulation of our problem proves to be quite useful in comparing the gradient of $u$ with the gradient of $u_R$, and in particular to obtain key gradient estimates (see Theorems \ref{stime gradiente}  and \ref{stime gradiente beta=0}). The main tool to obtain this result is the $p$-Bochner formula provided in \cite{Valtorta}, and a suitable maximum principle argument. The gradient estimates are then used to derive curvature bounds on the inner and outer boundary components of the domain $\Omega$, which in turn yield quantitative bounds on their lengths. Together with the rigidity statement associated with the equality case in the gradient estimates, these length bounds are then used in the final step of the proof to establish the rotational symmetry of the solution (see the concluding part of the proof of Theorem \ref{teo rigidità caso beta nullo}, Case 1, and Theorem \ref{teo rigidità caso beta nullo}, Case 2).
We also need length bounds for the curve $\Sigma$, which are again obtained from suitable curvature bounds. A fundamental tool in deriving these curvature estimates is the asymptotic expansion of $|\nabla u|^{p-1}$ and $|\dot u_R|^{p-1}$ near $\Sigma$. It is precisely at this stage that the asymptotic expansion of $u$ at $\Sigma$, given in \eqref{ass_esp}, becomes crucial to our argument (see Propositions \ref{lemma7.1} and \ref{prop stima VR}).\\

The paper is organized as follows. The fluid-dynamical interpretation of our problem is described in Section \ref{sez2}. In Section \ref{sez3}, we establish a Pohozaev-type identity, which will serve as a key tool to obtain some necessary conditions on the constants $M,\,a,\,b,\,\alpha,\,\beta$ such that a solution to problem \eqref{sist a b alpha beta} does exist. Section \ref{sez4} is devoted to introducing explicit ring-shaped model solutions for our problem and the comparison scheme. Section \ref{sez stime grad} contains the core gradient estimates, while Section \ref{sez6} contains a characterization of $p$-Serrin solutions in terms of the quantity $\max_{\partial\Om}\frac{|\nabla u|^p}{u_{\rm max}}$, so in particular the proof of Theorem \ref{teo2.1NEW}. In Sections \ref{sez7} and \ref{sez8}, we derive sharp length bounds for the boundary components of our domain, and for the curve $\Sigma$, respectively. Finally, Section \ref{sez9} is devoted to the proof of the rigidity results, that is Theorem \ref{teo rigidità a<b e a>b} and Theorem \ref{TEOREMA rigidita}. Throughout the paper, $|\cdot|$ will denote the Euclidean norm in $\R^2$.

%------------------------------------------------------------------

\section{Physical interpretation}\label{sez2}

In this section, we provide a description of the physical interpretation of our problem. In order to do that, let us first recall the Cauchy equation of motion, which is the fundamental balance law that expresses the conservation of linear momentum in continuum mechanics. It is valid for any continuum medium regardless of its  molecular nature, and reads as follows
\begin{equation}\label{cauchy eq}
\rho \dot v=\rho f+{\rm div }T
\end{equation}
where $\rho$ is the mass density, $f=f(t,X)$ is the external body force per unit mass, $v=v(t,X)$ is the velocity vector field, $\dot v$ represents the material derivative of the spatial vector field $v$ (that is $\dot v=\frac{\partial v}{\partial t}+(D_X v)v$), $T$ is the stress tensor (it is the constitutive law specifying $T$ that encodes the type of medium that one considers), and ${\rm div}$ is w.r.t. the spatial variable $X$.

By exploiting mass conservation ($\frac{\partial\rho}{\partial t}+{\rm div}(\rho v)=0$) and using the identity
\[
{\rm div}(\rho v\otimes v)=v{\rm div}(\rho v)+\rho (D_X v)v,
\]
equation \eqref{cauchy eq} can be rewritten as
\begin{equation}\label{cauchy eq 2}
\frac{\partial}{\partial t}(\rho v)+{\rm div}(\rho\,v\otimes v)=\rho f+{\rm div}T.
\end{equation}
We recall that for a  \textit{perfect fluid}, that is for fluids that do not exert appreciable tangential stresses, the constitutive equation for $T$ is 
\[T=-{\rm p}I,
\]
where ${\rm p}$ is the pressure field.
For a general viscous fluid, we recall from \cite{SerrinBook} that Stokes' principle holds, that it can be stated in form of postulates, and that such postulates lead to the following form of the stress tensor:
\[
T=\alpha I+\beta D+\gamma D^2,
\]
where $D:=\frac{1}{2}({\rm D}v+{\rm D}v^T)$ is the deformation tensor and $\alpha$, $\beta$ and $\gamma$ are scalar functions of the principal invariants of $D$.
From now on, let us assume that the fluid is  \emph{incompressible} (the volume of any fluid element is preserved), namely that 
\[
{\rm div}v=0.
\]
We also recall, en passant, that for a \textit{Newtonian fluid}
the stress tensor is an affine function of $D$, namely the \textit{Cauchy-Poisson} law holds:
\[
T=-{\rm p}I+\mu D,
\]
where $\mu$ is the \emph{dynamic viscosity}.
This law, coupled with \eqref{cauchy eq 2} yields the Navier-Stokes equation
\[
\frac{\partial}{\partial t}(\rho v)+{\rm div}(\rho\,v\otimes v)
=
\rho f-\nabla p+2\mu{\rm div}D.
\]
Among the \emph{non-Newtonian} fluids, we now consider the \emph{shear-dependent viscosity} fluids (also called \emph{generalised Newtonian} fluids, see, e.g., \cite{Diening}), namely those fluids such that 
\begin{equation}
\label{T_const}
T=-{\rm p}I+2\mu\big(|Dv|\big)\,D.
\end{equation}
Hence, unlike in the classical Newtonian setting, the dynamic viscosity of a
generalized Newtonian fluid may depend on the local shear rate $|Dv|$. For a
classical reference on the mathematical theory of viscous incompressible
flows, see \cite{Olga}. 
Let us go back to the Cauchy equation of motion with the constitutive assumption \eqref{T_const}. Supposing that the flow is taking place along a pipe of cross section $\Omega\ni X=(x,y)$, that such flow is steady and laminar (that is $v=\big(0,0,u(x,y)\big)$), and that $\rho(0,X)=\rho^*$, then from incompressibility and conservation of mass it follows that $\rho(t,X)=\rho^*$ $\forall t$, and \eqref{cauchy eq 2} reduces to
\[
f-\frac{\nabla {\rm p}}{\rho^*}+\frac{1}{\rho^*}\begin{pmatrix} 0\\ 0\\ \frac{1}{2}{\rm div}_{x,y}(\mu\nabla u) \end{pmatrix}	=0.
\]
If we finally assume 
\[
f=\begin{pmatrix} 0\\ 0\\ f_3=const.\end{pmatrix},
\]
we end up with
\begin{equation}
\label{modellino}
{\rm div}(\mu \nabla u)= c,
\end{equation}
where 
$\mu=\mu(|\nabla u|)$. 
Now, if we suppose that
\[
\mu(s)=K s^{p-2},
\qquad s=|\nabla u|,
\]
where \(K>0\) is a consistency parameter, the derived model \eqref{modellino}
coincides with the $p$-Laplace equation that is under consideration in the present paper.

When \(p>2\), then the dynamic viscosity increases as the shear rate increases.
This corresponds to \emph{shear-thickening} behaviour: the fluid becomes
more resistant to deformation under stronger shearing. A concentrated
mixture of cornstarch and water is a familiar example, since it flows
relatively easily under slow deformation but can become temporarily
stiff or almost solid-like when rapidly deformed. Regularity up to the
boundary for related full incompressible Navier-Stokes systems with
shear-dependent viscosity in the shear-thickening regime is studied in
\cite{veiga}.

When \(1<p<2\), then the dynamic viscosity decreases as the shear rate increases.
This corresponds to \emph{shear-thinning} behaviour: the fluid becomes
easier to deform under stronger shearing. Many paints exhibit this
behaviour, becoming less viscous while being brushed or rolled, which
allows them to spread easily, while remaining more viscous when nearly
at rest. Regularity up to the boundary
for related full incompressible systems in the shear-thinning regime is
investigated in \cite{Veiga0}.

Finally, when \(p=2\), the dynamic viscosity is independent of the shear
rate, since
$\mu(s)=K$.
Consequently, the nonlinear constitutive equation reduces, after normalisation, \[
\Delta u=c,
\]
which corresponds to Newtonian behaviour.

%-------------------------------------------------------------------

\section{Pohozaev-type identity}\label{sez3}

In this section, we derive a suitable Pohozaev-type identity for solutions to our problem. We then use such a relation to obtain some necessary conditions on the constants $M,\,a,\,b,\,\alpha,\,\beta$ such that a solution to problem \eqref{sist a b alpha beta} does exist.

\begin{theorem}
Let $\Omega$ be a bounded smooth domain in $\R^2$ and $p\in(1,+\infty)$. Assume that $u:\Omega\to\R$, $u\in C^2(\Omega)$  satisfies $\Delta_p u=-2$ in $\Omega$ and $u$ is locally constant on the boundary of the domain $\Omega$. The following identity holds
\begin{equation}\label{POHOZAEV}
\int_{\partial\Omega}\left[2u+\left(\frac{p-1}p\right)|\nabla u|^p\right] \langle X,\nu\rangle\,{\rm d}\sigma=
\int_\Omega \left(4u-\frac{2-p}p|\nabla u|^p\right){\rm d}\mu   
\end{equation}
where $\nu$ is the exterior unit normal to $\partial\Omega$ and $X:\Omega\to\R^2$ is the position vector $X(x)=(x_1,x_2)$.
\end{theorem}

\proof
Let us consider $f=f(x,u,v):\Omega\times\R\times\R^n \to\R$, $g:\Omega\to\R$ and $X:\Omega\to\R^n$ sufficiently regular and assume that $u:\Omega\to\R$, $u\in C^2(\Omega)$ satisfies
\begin{equation}\label{EL}
\mathrm{div}\Big({\rm D}_v f\,(x,u,\nabla u)\Big)\,=\,\partial_u f(x,u,\nabla u),\qquad x\in\Omega.    
\end{equation}
In \cite{PucciSerrin}, the following identity is derived
\begin{equation}\label{Pucci_Serrin}
\begin{split}
{\rm div}\Big(f(x,u,&\nabla u)X-\langle X,\nabla u\rangle\,
{\rm D}_v f(x,u,\nabla u)-g \, u{\rm D}_vf(x,u,\nabla u)\Big)\\
&=\,
f(x,u,\nabla u)\,{\rm div}X+\langle X,{\rm D}_x f(x,u,\nabla u)\rangle-\langle {\rm D}X\,{\rm D}_v f(x,u,\nabla u),\nabla u\rangle\\
&\quad-u\langle{\rm D}g,{\rm D}_vf(x,u,\nabla u)\rangle
-g\Big(\langle \nabla u,{\rm D}_v f(x,u,\nabla u)\rangle
+u\,\partial_uf(x,u,\nabla u)\Big),
\end{split}    
\end{equation}

If we suppose that
$$g\equiv0,\qquad\mbox{and}\qquad f(x,u,v)\,=\,f(u,v),$$
then \eqref{Pucci_Serrin} becomes
$${\rm div}\Big(f(u,\nabla u)X-\left\langle X,\nabla u\right\rangle
{\rm D}_v f(u,\nabla u)\Big)=f(u,\nabla u){\rm div}X-\left \langle {\rm D}X\,{\rm D}_vf(u,\nabla u),\nabla u\right\rangle.$$
and integrating over $\Omega$ we get
\begin{equation}\label{pucci ser integ}
\begin{aligned}
\int_{\partial\Omega}\big\langle f(u,\nabla u)X-\left\langle X,\nabla u\right\rangle {\rm D}_v &f(u,\nabla u)\,,\,\nu\big\rangle\,{\rm d}\sigma\\
&=\int_\Omega f(u,\nabla u){\rm div}X-\left \langle {\rm D}X\,{\rm D}_v f(u,\nabla u),\nabla u\right\rangle{\rm d}\mu.   
\end{aligned}
\end{equation}
Assuming also that
$$n=2, \qquad X(x)=x\qquad\mbox{and}\qquad
f(u,v)=\frac{|v|^p}p-2u,$$
then ${\rm D}_v f=|v|^{p-2}v$ and $\partial_u f=-2$, so the condition \eqref{EL} for $u$ becomes
$$\mathrm{div}\big(|\nabla u|^{p-2} \nabla u\big)\,=-2,\qquad x\in\Omega$$ 
that is 
$$\Delta_p u=-2.$$
Moreover, if $u$ is locally constant on $\partial\Omega$,
the outward unit normal $\nu$ at $\partial\Omega$ coincides with
$-\nabla u/|\nabla u|$. Putting together all the previous assumptions, the left-hand side of \eqref{pucci ser integ} becomes
$$\int_{\partial\Omega} \big\langle f(u,\nabla u)X-\langle X,\nabla u\rangle {\rm D}_vf(u,\nabla u)\,,\,\nu\big\rangle{\rm d}\sigma \,=\, \int_{\partial\Omega}\left[ \left(\frac{1-p}p\right)|\nabla u|^p-2u \right] \langle X,\nu\rangle{\rm d}\sigma$$
while the right-hand side is
$$\int_\Omega f(u,\nabla u){\rm div}X-\left\langle {\rm D}X \,{\rm D}_vf(u,\nabla u),\nabla u\right\rangle{\rm d}\mu
\,=\,\int_\Omega \left(\frac{2-p}p|\nabla u|^p-4u\right){\rm d}\mu$$
this proves our identity.
\endproof

Exploiting the previous Pohozaev identity, we obtain the following relations between the boundary values $a,\,b,\,\alpha,\,\beta$.

\begin{prop}\label{prop cond pohozaev}
Let $p\in(1,2)$ and $(\Omega,u)$ be a solution to problem \eqref{sist a b alpha beta}. The following inequalities hold.
\begin{itemize}
    \item [i.] If $a<b$ and $a\le u\le b$ in $\Omega$, then $\alpha<0$, $\beta\ge0$ and $2a+\frac{p-1}{p}|\alpha|^p> 2b+\frac{p-1}{p}\beta^p$.
    \item [ii.] If $a>b$ and $b\le u\le a$ in $\Omega$, then $\alpha\ge0$, $\beta<0$ and $2a+\frac{p-1}{p}\alpha^p> 2b+\frac{p-1}{p}|\beta|^p$.
\end{itemize}
\end{prop}

\proof
If $a<b$, then the solution $u$ decreases when approaching the inner boundary, and it increases when approaching the outer one, from which the signs for $\alpha$ and $\beta$ follow immediately. 
Actually, in order to rule out the case $\alpha=0$ (and $\beta=0$ in case $ii.$) we exploit the Hopf Lemma for the $p$-laplacian \cite[Proposition 3.2.1]{Tolk} (we can apply it since $u\in C^{1,\alpha}(\bar\Omega)$, see Remark \ref{regolarita soluz} below). Under the boundary conditions in \eqref{sist a b alpha beta}, the Pohozaev identity \eqref{POHOZAEV} becomes
$$\int\limits_{\Omega}\left(4u-\frac{2-p}{p}|\nabla u|^p\right)d\mu= \left(2a+\frac{p-1}{p}|\alpha|^p\right) \int\limits_{\Gamma_ \mathrm{i}}\langle X,\nu\rangle \,d \sigma+\left(2b+\frac{p-1}{p}|\beta|^p\right)\int\limits_{\Gamma_\mathrm{o}} \langle X,\nu\rangle\,d\sigma$$
and from the divergence theorem (recall ${\rm div}X=2$), we get
$$\int\limits_{\Omega}\left(4u-\frac{2-p}{p}|\nabla u|^p\right)d\mu=
-2\left(2a+\frac{p-1}{p}|\alpha|^p\right)|A_\mathrm{int}|+ 2\left(2b+\frac{p-1}{p}|\beta|^p\right)|A|$$
where $A,\,A_\mathrm{int}\subset\mathbb{R}^2$ are the two nonempty simply connected bounded domains such that $\Omega=A\setminus\overline{A_\mathrm{int}}$, furthermore $\Gamma_\mathrm{o}=\partial A$ and $\Gamma_\mathrm{i}=\partial A_\mathrm{int}$ are simple closed smooth curves. In particular, we have that
$|\Omega|=|A|-|A_\mathrm{int}|$. \\
\textit{Case i.} 
Isolating $|A_{\rm int}|$ it holds
$$\int\limits_{\Omega}\left(4(u-b)-\frac{2-p}{p}|\nabla u|^p-2\frac{p-1}{p}\beta^p\right)d\mu= \left(4b+2\frac{p-1}{p}\beta^p-4a-2\frac{p-1}{p}|\alpha|^p\right) |A_\mathrm{int}|,$$
by assumption, the LHS of the previous equality is negative, hence 
$2b+\frac{p-1}{p}|\beta|^p-2a-\frac{p-1}{p}|\alpha|^p< 0$.\\
\textit{Case ii.} Analogously, isolating $|A|$ it holds
$$\int\limits_{\Omega}\left(4(u-a)-\frac{2-p}{p}|\nabla u|^p-2\frac{p-1}{p}\alpha^p\right)d\mu=\left(4b+2\frac{p-1}{p}|\beta|^p-4a-2\frac{p-1}{p}\alpha^p\right)|A|,$$
we conclude as before.
\endproof

\begin{obs}
\label{regolarita soluz}
{\rm 
We collect here some regularity facts for solutions of
\begin{equation}\label{p-lapl=-2}
\Delta_p u:=\operatorname{div}\bigl(|\nabla u|^{p-2}\nabla u\bigr)=-2,
\end{equation}
in a smooth bounded domain $\Omega\subset\mathbb{R}^2$, with $1<p<2$.

By classical regularity results (see, e.g., \cite{Morrey}), any solution of \eqref{p-lapl=-2} is real analytic in the set
$\{x\in\Omega:\nabla u(x)\neq 0\}$.
At critical points, however, the operator becomes singular and analyticity may fail; in fact, solutions need not even be $C^2$ there.

Standard regularity theory (see \cite[Theorem 2]{DiBen} or \cite[Proposition 3.7]{Tolk}) yields
$u\in C^{1,\alpha}_{\rm loc}(\Omega)$,
for some $\alpha\in(0,1)$, with $\alpha\to0$ as $p\to1^+$. Boundary regularity for bounded weak solutions of both Dirichlet and Neumann problems was established by Lieberman. In particular, if $u$ solves \eqref{sist interno} or \eqref{sist esterno}, then $u\in C^{1,\alpha}(\overline{\Omega})$, for some $\alpha\in(0,1)$; see \cite[Theorem 1]{Lieberman}.

We also recall that the critical points of nonconstant $p$-harmonic functions in the plane are isolated; see \cite{Manfredi}. Finally, for solutions of
$\Delta_p u=f$,
with $f\in L^\infty$ and $p>2$, an optimal Hölder regularity result is available in \cite{AraTei}.
}
\end{obs}

For later use, we state in the following corollary the necessary conditions to have a solution to problem \eqref{sist interno} and to problem \eqref{sist esterno}, respectively. We skip the proof since it follows readily from Proposition \ref{prop cond pohozaev}.

\begin{corollary}\label{cor cond poho}
Assume that $p\in(1,2)$ and $M>0$.
Let $(\Omega,u)$ be a solution to problem \eqref{sist interno} such that $0\le u\le M$ in $\Omega$.
Then,
$$\alpha<0\qquad\text{and}\qquad  M<\frac{p-1}{2p}|\alpha|^p.$$ 
On the other hand, if  $(\Omega,u)$ is a solution to problem \eqref{sist esterno}. Then, 
$$\beta<0 \qquad\text{and}\qquad M> \frac{p-1}{2p}|\beta|^p$$
\end{corollary}

%------------------------------------------------------------------

\section{Preliminary results}\label{sez4}

\subsection{Model solutions for our problem}

First of all, let us write explicitly radial solutions to 
$$\Delta_p u=-2$$
over open ring-shaped domains $\Omega$ in $\R^2$.  For future convenience, we will distinguish between the cases where ring-shaped model solutions are strictly increasing or strictly decreasing in $\Omega$.\\

\textbf{Case 1.} We look for the pairs $(\Om,u)$, where $u=u(r)$ is a radial solution to $\Delta_p u=-2$ such that $u'(r)>0$ in the interior of $\Omega$, and the domain $\Om$ is a circular annulus in $\R^2$ of inner radius $r_i$ and outer radius $r_o$ (where $r_o>r_i>0$). 
It follows that we have to solve 
\[\frac{1}{r}\frac{d}{dr}\left(r(u')^{p-1}\right)= -2\qquad \text{for } r\in(r_i,r_o).
\]
Denoting by $w(r):=\left(u'(r)\right)^{p-1}$ we get the ODE
$$w'+\frac{1}{r}w=-2,$$
hence the solution is 
$$w(r)=\frac{R^2}{r}-r,\qquad\text{for }r\in(r_i,r_o)$$
where $R$ is a suitable constant such that $R^2\ge r_o^2$. 
Hence, the family of increasing rotationally symmetric solutions to $\Delta_p u=-2$, can be parametrized by the parameter $R$ such that $|R|\ge r_o$ and  it has the following form
\begin{equation}\label{sol mod cresc}
u_R(x)=u_R(r_i)+\int_{r_i}^{|x|}\left(\frac{R^2}{\rho}-\rho\right)^\frac{1}{p-1}d\rho
\end{equation}
over the domain
\begin{equation}\label{dom mod cresc}
\Omega_R=\{x\in\R^2\,|\,r_i<|x|<r_o\}.
\end{equation}
We will refer to the couple $(\Omega_R,u_R)$ as an \textit{inner model solution}. Notice that under this assumption on $R$ the solution $u_R$ is monotonically increasing inside the circular annulus $\Omega_R$ and it can have critical points only on the outer boundary component corresponding to $|x|=r_o$ when taking $R=r_o$.\\

\textbf{Case 2.} Arguing as in the previous case, we look for radial solutions to $\Delta_p u=-2$ over the circular annulus $\Om=\{x\in\R^2\,|\, r_i<|x|< r_o\}$ such that $u'(r)<0$. 
It follows that we have to solve 
$$\mathrm{div}\left((-u')^{p-2}\,u'\,\frac{\partial}{\partial r}\right)=-2\qquad \text{for } r\in(r_i,r_o).$$
Denoting by $w(r):=\left(-u'(r)\right)^{p-1}$, the previous becomes
$$w'+\frac{1}{r}w=2$$
hence 
$$w(r)=r+\frac{c_1}{r},\qquad\text{for }r\in(r_i,r_o)$$
where $c_1$ is a real negative constant such that $r+\frac{c_1}{r}>0$. For future convenience, we parametrize the family of decreasing rotationally symmetric solutions to $\Delta_p u=-2$ by the parameter $R>0$ such that $R^2<r_i^2$ in the following way
\begin{equation}\label{sol mod decr}
u_R(x)=u_R(r_i)-\int_{r_i}^{|x|}\left(\rho-\frac{R^2}{\rho}\right) ^\frac{1}{p-1}d\rho
\end{equation}
over the domain
\begin{equation}\label{dom mod decr}
\Omega_R=\{x\in\R^2\,|\,r_i<|x|<r_o\}.
\end{equation}
In this case, we will refer to the couple $(\Omega_R,u_R)$ as an \textit{outer model solution}.

\subsection{Comparison scheme}

To set up the comparison argument, we have to prove that for any suitable choice of the parameters $a,b,\alpha,\beta$ (according to the necessary conditions in Proposition \ref{prop cond pohozaev}) there exists a couple of the form \eqref{sol mod cresc}-\eqref{dom mod cresc} or of the form \eqref{sol mod decr}-\eqref{dom mod decr}, which solves problem \eqref{sist a b alpha beta}. The same argument applies to problem \eqref{sist interno} and problem \eqref{sist esterno}.

\begin{lemma}\label{lemma a<b}
Let us consider the boundary value problem \eqref{sist a b alpha beta} where we assume that $a<b$, $\alpha<0$, $\beta\ge0$ and 
$\left(\frac{p-1}{2p}\right)(|\alpha|^p-\beta^p)>b-a$.
Then, there exist unique $R,r_i,r_o$ real values such that $0<r_i<r_o\le R$ 
and the function 
$$u(x)=a+\int_{r_i}^{|x|} \left(\frac{R^2}{\rho}-\rho\right) ^\frac{1}{p-1}d\rho$$
solves problem \eqref{sist a b alpha beta} in the circular annular domain $\Omega=\{x\in\R^2\,|\,r_i<|x|<r_o\}$.
\end{lemma}

\proof 
Since by assumption $a<b$, we are looking for pairs of the form \eqref{sol mod cresc}-\eqref{dom mod cresc}, for some $R>0$. By imposing the boundary conditions, we get 
\begin{equation}
\label{sistema:slu_rot}
\begin{cases}
u(r_i)=a\\
b=a+\int\limits_{r_i}^{r_o}\left(\frac{R^2}{\rho}-\rho\right)^\frac{1}{p-1}d\rho\\
r_i^2+|\alpha|^{p-1}r_i-R^2=0\\
r_o^2+\beta^{p-1}r_o-R^2=0    
\end{cases}
\end{equation}
From the third and the fourth equation, since we look for positive values for $r_i$ and $r_o$, we get
\begin{equation}\label{r_int}
    r_i=\frac{-|\alpha|^{p-1}+\sqrt{\alpha^{2p-2}+4R^2}}{2}.
\end{equation}
\begin{equation}\label{r_out}
    r_o=\frac{-\beta^{p-1}+\sqrt{\beta^{2p-2}+4R^2}}{2}.
\end{equation}
It follows immediately that requiring $r_i<r_o$ is equivalent to asking for $\beta<|\alpha|$, which is true by assumption (recall that $\left(\frac{p-1}{2p}\right)(|\alpha|^p-\beta^p)>b-a$ since $b-a>0$). From \eqref{r_out} we also get that $r_o\le R$. 

Let us consider now the function $f:(0,+\infty)\to\R$ defined as follows
$$f(t):=\int\limits_{r_i(t)}^{r_o(t)}\left(\frac{t^2}{\rho}-\rho\right)^\frac{1}{p-1}d\rho+a-b$$
where $r_i(t)=\frac{1}{2}\left(-|\alpha|^{p-1}+\sqrt{\alpha^{2p-2}+4t^2}\right)$ and $r_o(t)=\frac{1}{2}\left(-\beta^{p-1}+\sqrt{\beta^{2p-2}+4t^2}\right)$. In order to conclude, we have to prove that $\exists\,R>0$ such that $f(R)=0$. Since the function $f$ is continuous, it is sufficient to prove that 
$$\lim_{t\to0^+}f(t)<0\qquad\text{and} \qquad \lim_{t\to+\infty}f(t)>0.$$
First, we note that 
$$\lim_{t\to0^+}f(t)=a-b<0.$$
Indeed, it holds
$$0\le\int\limits_{r_i(t)}^{r_o(t)}\left(\frac{t^2}{\rho}-\rho\right) ^\frac{1}{p-1}d\rho
\le t^\frac{2}{p-1}\int\limits_{r_i(t)}^{r_o(t)}\rho^{-\frac{1}{p-1}}
=\left(\frac{p-1}{2-p}\right)
t^2\left(\frac{1}{\beta^{p-2}}-\frac{1}{|\alpha|^{p-2}}\right) (1+o(1)),$$
which readily implies the above limit.
For what concerns the limit of $f$ as $t\to+\infty$, by changing variables and computing the Taylor expansions of $r_i$ and $r_o$, we obtain the following
$$\int\limits_{r_i(t)}^{r_o(t)}\left(\frac{t^2}{\rho}-\rho\right)^\frac{1}{p-1}d\rho= \int\limits_{\tau_1(t)}^{\tau_2(t)}\left(\frac{t^2}{t+\tau}-(t+\tau)\right)^\frac{1}{p-1}d\tau= 
\int\limits_{\tau_1(t)}^{\tau_2(t)}(-\tau)^\frac{1}{p-1}\left(\frac{2t+\tau}{t+\tau}\right)^\frac{1}{p-1}d\tau$$
where $\tau_1(t)=-\frac{|\alpha|^{p-1}}{2}+\frac{|\alpha|^{2p-2}}{8t}+o\left(\frac{1}{t}\right)$ and $\tau_2(t)=-\frac{\beta^{p-1}}{2}+\frac{\beta^{2p-2}}{8t}+o\left(\frac{1}{t}\right)$ as $t\to+\infty$.
By dominated convergence, it follows that 
$$\int\limits_{\tau_1(t)}^{\tau_2(t)}(-\tau)^\frac{1}{p-1}\left(\frac{2t+\tau}{t+\tau}\right)^\frac{1}{p-1}d\tau
\xrightarrow{t\to+\infty}\int\limits_{-\frac{|\alpha|^{p-1}}{2}} ^{-\frac{|\beta|^{p-1}}{2}}2^\frac{1}{p-1}(-\tau)^\frac{1}{p-1}\,d\tau=\,\frac{p-1}{p}\left(\frac{|\alpha|^p-\beta^p}{2}\right),$$
hence 
$$\lim\limits_{t\to+\infty}f(t)=\frac{p-1}{2p}\left( |\alpha|^p-\beta^p\right)+a-b\,>0$$
by assumption. We conclude that there exists $R>0$ such that $f(R)=0$.
We provide the proof of the uniqueness only in the simplified case where  $r_o=R$, $a=0$, $b=M$, $\beta=0$, the general case follows the same argument with minor modifications.
We compute
\begin{align}
f'(t)
&=
\frac{2t}{p-1}
\int_{r_i(t)}^t
\frac{1}{\rho}
\left(\frac{t^2}{\rho}-\rho\right)^{\frac{2-p}{p-1}}
d\rho
-
\frac{2t}{\sqrt{|\alpha|^{2p-2}+4t^2}}
\left(\frac{t^2}{r(t)}-r(t)\right)^{\frac{1}{p-1}}
\nonumber\\
&=
2t\left(
\frac1{p-1}
\int_{r(t)}^t
\frac{1}{\rho}
\left(\frac{t^2}{\rho}-\rho\right)^{\frac{2-p}{p-1}}
\,d\rho
-
\frac{|\alpha|}{\sqrt{|\alpha|^{2p-2}+4t^2}},
\right)
\label{per_unic}
\end{align}
where in the second identity we have used the third equation in \eqref{sistema:slu_rot}.
Now, performing the change of variable
$
s=\frac{t^2}{\rho}-\rho
$
in the integral term appearing in \eqref{per_unic}, we get
\begin{equation}
\label{id_deri_cresc}
\int_{r_i(t)}^t
\frac1{\rho}
\left(\frac{t^2}{\rho}-\rho\right)^{\frac{2-p}{p-1}}
d\rho
=
\int_0^{|\alpha|^{p-1}}
\!\!\frac{s^{\frac{2-p}{p-1}}}
{\sqrt{s^2+4t^2}}
ds.
\end{equation}
Also, not that
\[
\frac{d}{ds}
\left(
\frac{s^{\frac1{p-1}}}{\sqrt{s^2+4t^2}}
\right)
=
\frac1{p-1}
\frac{s^{\frac{2-p}{p-1}}}{\sqrt{s^2+4t^2}}
-
\frac{s^{\frac p{p-1}}}{(s^2+4t^2)^{3/2}}.
\]
Coupling this fact with identity \eqref{id_deri_cresc} and with expression \eqref{per_unic} yields
\[
f'(t)
=
2t
\int_0^{|\alpha|^{p-1}}
\frac{s^{\frac p{p-1}}}
{(s^2+4t^2)^{3/2}}
ds.
\]
In particular, we get that $f'>0$ and in turn the uniqueness statement.
\endproof

\begin{lemma}\label{lemma a>b}
Let us consider the boundary value problem \eqref{sist a b alpha beta} where we assume that $a>b$, $\alpha\ge0$, $\beta<0$ and $a-b>\left(\frac{p-1}{2p}\right)(|\beta|^p-\alpha^p)$. 
\begin{itemize}
    \item[$i)$] If $\alpha\ge|\beta|$, then problem \eqref{sist a b alpha beta} does not admit solutions of the form \eqref{sol mod decr}-\eqref{dom mod decr}.
    \item[$ii)$] Assuming $\alpha<|\beta|$, in the case when $a-b> \left(\frac{p-1}{p}\right)(|\beta|^p-\alpha^p)$ problem \eqref{sist a b alpha beta} does not admit solutions of the form \eqref{sol mod decr}-\eqref{dom mod decr}; while if $a-b\le\left(\frac{p-1}{p}\right)(|\beta|^p-\alpha^p)$ a solution of the form \eqref{sol mod decr}-\eqref{dom mod decr} does exist and it is unique. More precisely, there exist $R,r_i,r_o$ real values such that $0<|R|\le r_i<r_o$ and the function 
    $$u(x)=a-\int_{r_i}^{|x|} \left(\rho-\frac{R^2}{\rho}\right) ^\frac{1}{p-1}d\rho$$
solves problem \eqref{sist a b alpha beta} in the circular annular domain $\Omega=\{x\in\R^n\,|\,r_i<|x|<r_o\}$.
\end{itemize}
\end{lemma}

\noindent This result proves that the further assumption $a-b\le\left(\frac{p-1}{p}\right)(|\beta|^p-\alpha^p)$ is not restrictive for our problem.

\proof
Arguing as in the proof of Lemma \ref{lemma a<b}, after imposing the boundary conditions, we are looking for solutions of the form
$$u_R(x)=a-\int_{r_i}^{|x|}\left(\rho-\frac{R^2}{\rho}\right) ^\frac{1}{p-1}d\rho\qquad\text{in }\,\,\Omega_R=\{r_i<|x|<r_o\}$$
where $R>0$
$$r_i(R)=\frac{\alpha^{p-1}+\sqrt{\alpha^{2p-2}+4R^2}}{2}$$
$$r_o(R)=\frac{|\beta|^{p-1}+\sqrt{\beta^{2p-2}+4R^2}}{2}$$
and moreover
$$b=a-\int\limits_{r_i}^{r_o}\left(\rho-\frac{R^2}{\rho}\right) ^\frac{1}{p-1}d\rho.$$
Notice that $|R|\le r_i$, so there are no problems in the previous definitions. Moreover, $r_i<r_o$ if and only if $\alpha<|\beta|$, hence statement $i)$ follows at once.

From now on, let us assume that $\alpha<|\beta|$. We introduce the function $f:[0,+\infty)\to\R$ as follows
$$f(t):=a-b-\int\limits_{r_i(t)}^{r_o(t)}\left(\rho-\frac{t^2}{\rho}\right) ^\frac{1}{p-1}d\rho.$$
In this case, we can compute its value at $t=0$  
$$f(0)=a-b-\int_{\alpha^{p-1}}^{|\beta|^{p-1}}\rho^\frac{1}{p-1}d\rho=a-b-\frac{p-1}{p}(|\beta|^p-\alpha^p).$$
For what concerns the derivative of $f$ we have
$$f'(t)=-\left(r_o-\frac{t^2}{r_o}\right)^\frac{1}{p-1}r_o'+\left(r_i-\frac{t^2}{r_i}\right)^\frac{1}{p-1}r_i'+\int\limits_{r_i}^{r_o} \frac{2t}{(p-1)\rho}\left(\rho-\frac{t^2}{\rho}\right)^\frac{2-p}{p-1}d\rho.$$
Denoting by $G$ the primitive of $g(\rho,t):=\frac{2t}{(p-1)\rho}\left(\rho-\frac{t^2}{\rho}\right)^\frac{2-p}{p-1}$ (so that $\partial_\rho G=g$) and by $l_s(t):=\frac{s^{p-1}+\sqrt{s^{2p-2}+4t^2}}{2}$ we get
$$f'(t)=-\left(l_{|\beta|}-\frac{t^2}{l_{|\beta|}}\right)^\frac{1}{p-1}l_{|\beta|}'+G(l_{|\beta|},t)+\left(l_\alpha-\frac{t^2}{l_\alpha}\right)^\frac{1}{p-1}l_{\alpha}'-G(l_\alpha,t),$$
hence
$$f'(t)=h(\alpha,t)-h(|\beta|,t)$$
where $$h(s,t):=\left(l_s-\frac{t^2}{l_s}\right)^\frac{1}{p-1}l_s'-G(l_s,t).$$
So, in order to study the monotonicity of $f$ we study the monotonicity of $h$ in the first variable. We obtain
$$\frac{d}{ds} h(s,t)=-\left(l_s-\frac{t^2}{l_s}\right)^\frac{1}{p-1} \frac{(p-1)\,t\,s^{3p-4}}{(l_s^2-t^2)\,l_s\,(s^{2p-2}+4t^2)^{3/2}} \left[s^{2p-2}+\frac{2p}{p-1}t^2+s^{p-1}\sqrt{s^{2p-2}+4t^2}\right]$$
which is clearly negative when assuming $s,\,t>0$. Since by assumption $\alpha<|\beta|$, this proves that $f'(t)>0$.
If $a-b> \left(\frac{p-1}{p}\right)(|\beta|^p-\alpha^p)$, we get $f(0)>0$, and since $f$ is monotone increasing, then $\not\exists R>0$ such that $f(R)=0$, so our problem does not admit solutions of the form \eqref{sol mod decr}-\eqref{dom mod decr}.

On the other hand, in the case when $a-b\le\left(\frac{p-1}{p}\right)(|\beta|^p-\alpha^p)$, we have
$$f(0)\le0$$
and
$$\lim_{t\to+\infty}f(t)>0,$$
since the function $f$ is continuous and strictly increasing, then there exists a unique value $R\in(0,r_i)$ such that $f(R)=0$. 
\endproof

Notice that Lemma \ref{lemma a<b} and Lemma \ref{lemma a>b} cover in particular the case of the boundary value problems \eqref{sist interno} and \eqref{sist esterno}. For the sake of simplicity, we recall the results in the following two corollaries.

\begin{corollary}\label{lemma a<b caso beta=0}
If $\alpha$ and $M$ are real constants such that $\alpha<0$ and $0<M<\left(\frac{p-1}{2p}\right)|\alpha|^p$.
Then, there exist $R$ and $r_i$ real values, such that $0<r_i<R$ and the function 
$$u(x)=\int_{r_i}^{|x|} \left(\frac{R^2}{\rho}-\rho\right) ^\frac{1}{p-1}d\rho$$
solves problem \eqref{sist interno} in the circular annular domain $$\Omega=\{x\in\R^2\,|\,r_i<|x|<R\}.$$    
\end{corollary}

\begin{corollary}\label{lemma a>b alfa=0}
Let $\beta$ and $M$ be real constants such that $M>\frac{p-1}{2p}\,|\beta|^p>0$.
\begin{itemize}
    \item[$i.$] If $M\ge\frac{p-1}{p}\,|\beta|^p$ problem \eqref{sist esterno} does not admit solutions of the form \eqref{sol mod decr}-\eqref{dom mod decr}.
    \item[$ii.$] If $\frac{p-1}{2p}\,|\beta|^p<M<\frac{p-1}{p}\,|\beta|^p$, then there exist $R$ and $r_o$ real values such that $0<R<r_o$ and the function 
    $$u(x)=M-\int_{R}^{|x|} \left(\rho-\frac{R^2}{\rho}\right) ^\frac{1}{p-1}d\rho$$
    solves problem \eqref{sist esterno} in the circular annular domain 
    $$\Omega=\{x\in\R^n\,|\,R<|x|<r_o\}.$$
\end{itemize}
\end{corollary}

%--------------------------------------------------------------------

\section{Gradient estimates}\label{sez stime grad}

Given $(\Omega,u)$ a solution to problem \eqref{sist a b alpha beta}, from Proposition \ref{prop cond pohozaev} and Lemma \ref{lemma a<b} or Lemma \ref{lemma a>b} (depending on whether $a<b$ or $a>b$), we can associate to $(\Om,u)$ three positive real parameters $R,\,r_i,\,r_o$, functions of the Dirichlet data $a,\,b$ and of the Neumann data $\alpha,\,\beta$, that we will call \textit{comparison parameters}. Consequently, we can associate to $(\Om,u)$ the model solution $(\Om_R,u_R)$. 

In particular, if $(\Omega,u)$ is a solution to problem \eqref{nostro problema}, considering its restriction to $\Om_{\rm i}$, we get a solution to system \eqref{sist interno} for $M=u_{\rm max}$. Then, from Corollary \ref{cor cond poho} together with Corollary \ref{lemma a<b caso beta=0}, we can associate to $(\Om_{\rm i},u_{|\Om_{\rm i}})$ two positive real parameters $r_i$ and $R$ (functions of $u_{\rm max}$ and $\alpha$). Consequently, we can associate to $(\Om,u)$ restricted to $\Om_{\rm i}$, the model solution $(\Om_R,u_R)$. Analogously, using Corollary \ref{lemma a>b alfa=0}, we associate to $(\Om_{\rm o},u_{|\Om_{\rm o}})$ the two values $R$ and $r_o$ (functions of $u_{\rm max}$ and $\beta$). 

In this section, we compare the gradient of $(\Om,u)$ solution to problem \eqref{sist a b alpha beta} with the gradient of the associated ring-shaped model solution $(\Om_R,u_R)$.

\subsection{Case 1: internal subdomain.}

If $a<b$ and $(\Om,u)$ is a solution to problem \eqref{sist a b alpha beta}, then the comparison parameters are such that $(0<)\,r_i<r_o\le R$ and $u_R:[r_i,r_o]\longrightarrow\R$ is defined as
\[u_R(r)\,:=\,a+\,\int_{r_i}^r\left(\frac{R^2}{\rho}-\rho\right)^\frac{1}{p-1}d\rho.\]
Note that $u_R$ is strictly increasing and that $u_R\big([r_i,r_o]\big)=[a,b]$.
We define the function $\psi_i:[a,b]\to[r_i,r_o]$
as the inverse of $u_R$. So we get
\begin{equation}\label{psi' e psi''}
\dot\psi_i=\left(\frac{\psi_i}{R^2-\psi_i^2}\right)^\frac{1}{p-1}\qquad \text{and}\qquad \ddot\psi_i=\frac{2}{p-1}\dot\psi_i^{p+1} +\frac{1}{p-1}\frac{\dot\psi_i^2}{\psi_i}    
\end{equation}
The \textit{pseudo-radial function} $\Psi:\Om\to[r_i,r_o]$ is defined as
$$\Psi(x):=\psi_i\left(u(x)\right).$$
We introduce the following \textit{ansatz metric}
$$g=\frac{1}{\Psi^2}\,g_{\R^2}$$
which clearly is conformal to the Euclidean metric $g_{\R^2}=dr\otimes dr+r^2 d\theta\otimes d\theta$, and it coincides with the cylindrical metric when $\Psi(x)=r(x)$. We also introduce the \textit{pseudo-affine function} \begin{equation}
\label
{filogtri}
\varphi(x):=\log\left(\psi_i(u(x))\right),
\end{equation}
note that $|\nabla\varphi|_g=1$ in the radially symmetric case.
Denoting by $\nabla$ and $\Delta_g$ the Levi-Civita connection and the Laplace-Beltrami operator of $g$, respectively, and by $\Delta^{(g)}_p$ the $p$-Laplacian with respect to $g$, we can rephrase our problem in terms of the new metric $g$.

\begin{prop}\label{pb phi interno}
Let $a<b$ and $(\Om,u)$ be a solution to problem \eqref{sist a b alpha beta}. Then, the metric $g$ and the pseudo-affine function $\f$ satisfy the following system over $\Omega$
\begin{equation}\label{sist g e fi}
\begin{cases}
\Delta_p^{(g)}\f=2\psi_i\dot\psi_i^{p-1}\Big(|\N\f|^p_g-1\Big)\\
    \mathrm{Ric}_g=(\Delta_g \f)\,g,    
\end{cases}
\end{equation}
where the functions $\psi_i$ and $\dot\psi_i$ are meant to be evaluated at $u(x)$.
\end{prop}

\proof
Since $u$ solves problem \eqref{sist a b alpha beta}, we have
\begin{equation}\label{id p lapl}
-2=\Delta_p u=\mathrm{div}(|\D u|^{p-2}\,\D u)=\langle\D(|\D u|^{p-2}),\D u\rangle\,+\, |\D u|^{p-2}\Delta u. 
\end{equation}
We rewrite both terms on the right-hand side in terms of $\f$ and $g$. To simplify the notation, we write 
$\psi$ in place of $\psi_i$. For the first term, it holds
$$\langle\D(|\D u|^{p-2}),\D u\rangle=\frac{1}{\psi\dot\psi^{p-1}} \langle\N(|\N\f|^{p-2}),\N\f\rangle_g+(2-p)\frac{\ddot\psi}{\dot\psi^{p+1}}|\N\f|_g^p.$$
On the other hand, to deal with the second term we recall that for every $z\in C^\infty$ we have
$$\nabla_i\nabla_j z=\D_i\D_j z+\frac{1}{\psi}\bigg(\partial_i z\,\partial_j\psi+\partial_i\psi\,\partial_j z- \langle\D z,\D \psi\rangle g^{\R^2}_{ij} \bigg) $$
from which we get
$$\Delta_g\f=\psi\dot\psi\Delta u+\psi^2\langle\, \D\Big(\frac{\dot\psi}{\psi} \Big),\D u\rangle$$
isolating $\Delta u$ and rewriting in terms of the new metric, we finally get
$$\Delta u=\frac{1}{\psi\dot\psi}\Delta_g\f-\frac{1}{\dot\psi^2} \langle\, \D\Big(\frac{\dot\psi}{\psi} \Big),\N\f\rangle_g.$$
Plugging in \eqref{id p lapl}, we obtain
$$\Delta_p^{(g)}\f=2\psi\dot\psi^{p-1}\Big(|\N\f|^p_g-1\Big).$$
Exploiting the formula for the conformal change of the Ricci tensor, and  
the fact that $\mathrm{Ric}_{g_{\R^n}}\equiv 0$ we obtain
$$\mathrm{Ric}_{g}=(\Delta_g\f)\,g.$$
\endproof

Let us set 
\begin{equation}
\label{defWW_R}
W:=|\nabla u|^p\qquad\text{and}\qquad W_R:=|\nabla u_R|^p\circ\Psi=\left(\frac{R^2-\Psi^2}{\Psi}\right)^\frac{p}{p-1},
\end{equation}
and note that 
\begin{equation}\label{w su wr nabla phi}
    |\N\varphi|_g^p=\frac{W}{W_R}.
\end{equation}
Taking advantage of the previous equivalent reformulation of our problem, we obtain the following gradient estimates, we also provide our first rigidity result.
More precisely, since the proof is substantially different and relies on distinct assumptions, we distinguish between the case $\beta \neq 0$ and the case $\beta = 0$.

\begin{theorem}\label{stime gradiente}
Let us fix $a<b$ and $\beta\not=0$. Let $(\Om,u)$ be a solution to problem \eqref{sist a b alpha beta} and $R$ be the comparison parameter given by Lemma \ref{lemma a<b}. Then, 
\begin{equation}\label{W<=WR}
    W\le W_R,\quad\text{in }\Omega.
\end{equation}
Moreover, if there exists $x\in\Omega$ such that $W(x)=W_R(x)$, then $(\Om,u)$ coincides with $(\Om_R,u_R)$, up to translations.  
\end{theorem}

\proof 
The main tool in this proof is the $p$-Bochner formula established in \cite{Valtorta}. Following the notation therein, we consider the differential operator
\begin{equation}
\label{p_II:def}
P^{II}_\f(f)\,:=\,
|\N\f|^{p-2}\Delta f+ (p-2)|\N\f|^{p-4}\N^2 f\,(\N\f,\N\f),
\end{equation}
which is the second-order operator appearing in the linearisation of the 
$p$-Laplacian at $\f$. Since $|\N\f|_g>0$ in $\Omega$, the $p$-Bochner formula \cite[Proposition 3.1.2]{Valtorta}
says that
\begin{equation}\label{p bochner fi}
\begin{aligned}
P^{II}_\f\Big(|\N\f|_g^p\Big)=p\,|\N \f|_g^{2p-4}&\Bigg\{ |\N\f|_g^{2-p}\bigg[\langle\N\Delta_p^{(g)} \f\,,\N \f\rangle_g -(p-2)\frac{\N^2 \f(\N\f,\N\f)}{|\N \f|_g^2}\Delta_p^{(g)} \f\bigg]\\
&+|\N^2\f|_g^2\,+\,p(p-2)\frac{\big(\N^2 \f(\N\f,\N\f)\big)^2}{|\N \f|_g^4}+\mathrm{Ric}_g(\N\f,\N\f) \Bigg\},
\end{aligned}
\end{equation}
throughout $\Omega$.
Using equations \eqref{sist g e fi} and setting 
\[
w\,:=\,|\N\f|_g^p-1,
\]
after some calculations, we obtain 
\begin{align*}
\mathrm{Ric}_g(\N\f,\N\f)
&\,=\,
2\psi\,\dot\psi^{p-1}|\N\f|_g^{4-p}\,w\,-\,\Big(\frac{p-2}{p}\Big) |\N\f|_g^{2-p}\langle\N w,\N\f\rangle_g,\\
\langle\N\big(\Delta_p^{(g)}\f\big),\N\f\rangle_g
&\,=\,
2\Big(\psi\,\dot\psi^{p-1}+(p-1)\psi^2\dot\psi^{p-3}\ddot\psi\Big)|\N\f|_g^2\,w\,+\, 2\psi\dot\psi^{p-1}\langle\N w,\N\f\rangle_g.
\end{align*}
Here and throughout the proof, the function $\psi$ stands for $\psi_i$ and is always meant, together with its derivatives, to be evaluated at $u(x)$.
Substituting into \eqref{p bochner fi} and using the identity
\begin{equation}\label{espr  hess fi}
\N^2\f(\N\f,\N\f)=\frac{1}{p}\,|\N\f|_g^{2-p}\langle\N w,\N\f\rangle_g,
\end{equation}
we get
\begin{equation}
\label{stima ell int}
P^{II}_\f(w)\,+\, a\langle\N w,\N\f\rangle_g
\,-\,2p\,\Big[2\psi\,\dot\psi^{p-1}+(p-1)\psi^2\dot\psi^{p-3}\ddot\psi\Big]|\N\f|_g^{p}\,w
=\, 
p|\nabla\f|_g^{2p-4}|\nabla^2\f|_g^2,
\end{equation}
where
\[a\,=\,-\left(2-p+2p\,\psi\,\dot\psi^{p-1}\right)\!|\N\f|_g^{p-2}
-2(2-p)\psi\,\dot\psi^{p-1}|\N\f|_g^{-2}w
-(p-2)|\N\f|_g^{-4}\langle\N w,\N\f\rangle_g.\]
Hence
$$P^{II}_\f(w)\,+\,a\,\langle\N w,\N\f\rangle_g +c\,w\,\ge0,$$
using the formula for $\ddot\psi$ in \eqref{psi' e psi''}, we get
$$c=-p\left[6\psi\dot\psi^{p-1}+4\psi^2\dot\psi^{2p-4}\right] |\N\f|_g^{p}$$
which is non-positive since $\psi>0$ and $\dot\psi\ge0$.
We can thus invoke the Weak Maximum Principle (see Remark \ref{obs max pr}), which gives
$$\sup_{\Om} w\le\sup_{\partial\Om}w^+.$$
To conclude, we have to study the behaviour of $w$ on $\partial \Om=\Gamma_{\rm i}\sqcup\Gamma_{\rm o}$.
However, the function $w$ is not a priori well defined on the whole boundary of $\Omega$, but only on $\Gamma_{\rm i}$. More precisely, recalling that
$\Psi_{|_{\Gamma_{\rm i}}}=\psi(a)=r_i$,
we have that
\begin{equation}\label{calc_sigma_in}
{|\N\f|_g}_{\big|_{\Gamma_{\rm i}}}=\dot\psi(a)|\D u|_{\big|_{\Gamma_{\rm i}}}=-\frac{\alpha\,\,r_i^\frac{1}{p-1}}{(R^2-r_i^2)^\frac{1}{p-1}}=1,
\end{equation}
so that $w=0$ on $\Gamma_{\rm i}$. On the other hand, $\Psi_{|_{\Gamma_{\rm o}}}=\psi(b)=r_o$, so
$${|\N\f|_g}_{\big|_{\Gamma_{\rm o}}}=
\dot\psi(b)|\D u|_{\big|_{\Gamma_{\rm o}}}=
\frac{\beta\,\,r_o^\frac{1}{p-1}}{(R^2-r_o^2)^\frac{1}{p-1}}$$
which is not well-defined in the case when $\beta=0$, being $R=r_o$.
Since we are assuming $\beta>0$, then $R>r_o$, hence $|\N\f|_g=1$ and $w=0$ on $\Gamma_{\rm o}$.
We can conclude that $w\le0$ in $\Omega$, namely $|\N\f|^p_g\le1$ in $\Omega$, and recalling the definition of the pseudo-affine function $\f$, this means $W\le W_R$ in $\Om$.

Now, suppose that there exists $x\in\Omega$ such that $W(x)=W_R(x)$. Hence, $w(x)=0$ and in turn $w=0$ in $\Om$, by the Strong Maximum Principle (see Remark \ref{obs max pr}). This fact implies that $|\N\f|_g\equiv 1$ in $\Omega$ and in turn that $|\nabla^2\f|_g\equiv0$ in $\Omega$, by \eqref{stima ell int}. Therefore, the second fundamental form $h$ of the level sets of $\f$ is null. Since $h$ is related with the second fundamental form $\tilde h$ of the level sets of $u$ - computed w.r.t. $g_{\R^2}$ - by the formula
\[
{h_{ij}}_{\big|_{\{\f=\log(\psi(c))\}}}={\frac{\tilde{h}_{ij}}{\psi}}_{\big |_{\{u=c\}}}\!\!-\frac{\dot\psi}{\psi^2}|{\rm D}u|\,{g_{ij}^{\R^2}}_{\big|_{\{u=c\}}},
\]
we obtain that all the level sets of $u$ are totally umbilic curves in $\R^2$. In turn,  being closed the level sets are isometric to circles, proving that $u$ is radially symmetric.
\endproof

The analogous gradient estimate in the case $\beta=0$ requires a reverse {\L}ojasiewicz inequality, according to the following definition. 

\begin{definition}[$\bar\gamma$-reverse {\L}ojasiewicz inequality]
\label{RRLI}
Let $(M,g)$ be a Riemannian manifold,
$f:M\to\R$, and let $\bar\gamma>0$.
We say that $f$ satisfies a \emph{$\bar\gamma$-reverse {\L}ojasiewicz inequality} at ${\rm MAX}(f)$ if for every $\gamma<\bar\gamma$ and for any neighbourhood $U$ of ${\rm MAX}(f)$,
there exists a real number $c>0$ such that
\[|\nabla f|(x)\le c\big[f_{\rm max}-f(x)\big]^\gamma,\]
for every $x\in U$.
\end{definition}

\noindent
We recall that, if $f$ is smooth, the reverse {\L}ojasiewicz inequality holds with $\bar\gamma=1/2$ (see \cite[Theorem 2.2]{BoChMa}).

\begin{theorem}\label{stime gradiente beta=0}
Assume $a<b$ and $\beta=0$. Let $(\Om,u)$ be a solution to problem \eqref{sist a b alpha beta} and $R$ the comparison parameter given by Lemma \ref{lemma a<b}. Assume further that $u$ satisfies a $\frac{1}{p^2}$-reverse {\L}ojasiewicz inequality at $\Gamma_\mathrm{o}$.
Then
\begin{equation*}
    W\le W_R,\quad\text{in }\Omega.
\end{equation*}
Moreover, if there exists $x\in\Omega$ such that $W(x)=W_R(x)$, then $(\Om,u)$ coincides with $(\Om_R,u_R)$, up to translations.  
\end{theorem}

\proof
As pointed out in the previous proof, if $\beta=0$, $w$ is not well-defined on $\Gamma_{\rm o}$, in particular, $\frac{\partial u}{\partial\nu}=0$ on $\Gamma_{\rm o}$ so the outer boundary component is a curve of maximum points.
To overcome this difficulty we introduce an auxiliary function $\eta$, that is we rewrite the equation \eqref{stima ell int} in terms of
\[F\,:=\, \eta(\f) w\]
where $\eta$ is a positive auxiliary function to be chosen appropriately.
First of all, it is easy to see that
\begin{align*}
P^{II}_\f(F) & = |\nabla\f|^{p-2}\Delta_g F+(p-2)|\nabla\f|^{p-4}\nabla^2 F(\nabla\f,\nabla\f)\\
&=\eta |\nabla\f|^{p-2}\Delta_g w+2\dot\eta|\nabla\f|^{p-2}\langle\nabla w,\nabla\f\rangle+\ddot\eta w|\nabla \f|^p+\dot\eta w|\nabla\f|^{p-2}\Delta_g \f\\
&\quad+(p-2)\eta|\nabla\f|^{p-4}\nabla^2 w(\nabla\f,\nabla\f)+2(p-2)\dot\eta|\nabla\f|^{p-2}\langle\nabla w,\nabla\f\rangle\\
&\quad+(p-2)\ddot\eta\, w|\nabla\f|^p+(p-2)\dot\eta\,w\,|\nabla\f|^{p-2}\nabla^2\f(\nabla\f,\nabla\f)\\
&=\eta P^{II}_\f(w)+2(p-1)\dot\eta|\nabla\f|^{p-2}\langle\nabla w,\nabla\f\rangle+(p-1)\ddot\eta w|\N\f|^p+ 2\psi\dot\psi^{p-1}\dot\eta w^2
\end{align*}
where in the last equality we used that 
\begin{equation}\label{lapl g fi}
\Delta_g\f=2\psi\dot\psi^{p-1}w|\nabla\f|^{2-p}-\frac{p-2}{p}|\nabla\f|^{-p}\langle\nabla w,\nabla\f\rangle    
\end{equation}
and \eqref{espr  hess fi}. Now, replacing $P_\f^{II}(w)$ with the expression provided in \eqref{stima ell int}, and rewriting in terms of $F$, after some calculations we finally get
\begin{align}\notag
&P^{II}_\f(F)=\left[2p\frac{\psi^2(3R^2-\psi^2)}{(R^2-\psi^2)^2}-\frac{\dot\eta}{\eta}\left(2-p+2p\frac{\psi^2}{R^2-\Psi^2}\right) +(p-1)\frac{\ddot\eta}\eta-2(p-1)\left(\frac{\dot\eta}\eta\right)^2\right]
|\N\f|_g^p \,F\\\notag
&\quad+\left[2-p+2p\frac{\psi^2}{R^2-\psi^2}+2(p-1)\frac{\dot\eta}{\eta}\right]|\N\f|^{p-2}\langle\N F,\N\f\rangle+p\eta|\N\f|^{2p-4}|\N^2\f|^2\\  \label{Pii}
&\quad-\frac{2(p-2)}{\eta}\left[\frac{\psi^2}{R^2-\psi^2}+\frac{\dot\eta}{\eta}\right] \frac{F}{|\N\f|^2}\langle\N F,\N\f\rangle+\left[2(p-1)\frac{\psi^2}{R^2-\psi^2}+(p-2)\frac{\dot\eta}{\eta}\right]\frac{\dot\eta}{\eta^2}\,F^2.
\end{align}
By exploiting the estimate $|\nabla^2\f|_g^2\ge\frac{1}{2}(\Delta_g\f)^2$ together with \eqref{lapl g fi}, we finally get
\begin{equation*}
\begin{aligned}
P^{II}_\f(F)&\ge\left[(p-1)\frac{\ddot\eta}{\eta}-2(p-1)\left(\frac{\dot\eta}{\eta}\right)^2-\left(2-p+2p\psi\dot\psi^{p-1}\right)\frac{\dot\eta}{\eta}+2p\psi^2\frac{3R^2-\psi^2}{(R^2-\psi^2)^2}\right]|\nabla\f|_g^p F\\
&+\left[2-p+2p\psi\dot\psi^{p-1}+2(p-1)\frac{\dot\eta}{\eta} \right]|\nabla\f|_g^{p-2}\langle\nabla F,\nabla\f\rangle_g\\
&+\frac{2(2-p)}{\eta}\left[2\psi\dot\psi^{p-1}+\frac{3p-2}{2p}\frac{\dot\eta}{\eta}\right]\,F\, \frac{\langle\nabla F,\nabla\f\rangle_g}{|\nabla\f|_g^2}+\frac{1}{\eta}\left[p-2+ \frac{(p-2)^2}{2p}\right]\frac{\langle\nabla F,\nabla\f\rangle_g^2}{|\nabla\f|_g^4}\\
&+\frac{1}{\eta}\left[2\psi\dot\psi^{p-1}\left((2p-3)\frac{\dot\eta}{\eta}+p\psi\dot\psi^{p-1}\right)+(p-2)\frac{3p-2}{2p}\left(\frac{\dot\eta}{\eta}\right)^2\right] F^2.
\end{aligned}   \end{equation*}
Let us now rewrite the above inequality as
\[
P^{II}_\f(F)+a\,\langle\nabla F,\nabla\f\rangle_g+ c\, F\ge d F^2,
\]
where
\begin{align*}
c\,&:=\,
-\left[(p-1)\left(\frac{\dot\eta}{\eta}\right)'-(p-1)\left(\frac{\dot\eta}{\eta}\right)^2-\left(2-p+2p\psi\dot\psi^{p-1}\right)\frac{\dot\eta}{\eta}+p\frac{2\psi^2(3R^2-\psi^2)}{(R^2-\psi^2)^2}\right]|\nabla\f|_g^p,\\
d\,&:=\,\frac{1}{2p\,\eta}\left[(p-2)(3p-2)\left(\frac{\dot\eta}{\eta}\right)^2+4p\,(2p-3)\psi\dot\psi^{p-1}\frac{\dot\eta}{\eta}+4p^2\, \psi^2\dot\psi^{2p-2}\right].
\end{align*}
Let us set 
\begin{equation}\label{ode eta}
\frac{\dot\eta(\f)}{\eta(\f)}=m\psi\dot\psi^{p-1},   
\end{equation}
for some $m\in\R$ to be chosen later, and note that, since $\psi(u(x))=\Psi(x)=e^{\f(x)}$, the above ODE is solved by $\eta(\f)=(R^2-e^{2\f})^{-\frac{m}{2}}$ (which is positive inside the domain $\Omega$).
Using \eqref{ode eta}, the expressions of the coefficients $c$ and $d$ reduce to
\begin{align*}
c\,&=\,\frac{\psi^2}{(R^2-\psi^2)^2}\left[\Big((p-1)m^2+(3p-2)m+2p\Big)\psi^2-\Big((3p-4)m+6p\Big) R^2\right]|\nabla\f|_g^p,\\
d\,&=\frac{1}{2p\,\eta}\Big[(p-2)(3p-2)m^2+4p\,(2p-3) m+4p^2\,\Big] \psi^2\dot\psi^{2p-2}.
\end{align*}
It turns out that $d\geq0$ if and only if $m_1(p)\leq m\leq m_2(p)$ where $m_1(p)=\frac{-2p}{2p-3+\sqrt{p^2-4p+5}}$ and $m_2(p)=\frac{-2p}{2p-3-\sqrt{p^2-4p+5}}$ (note that $2p-3+\sqrt{p^2-4p+5}>0$ while $2p-3-\sqrt{p^2-4p+5}<0$ for any $p\in(1,2)$).
Therefore, a sufficient condition for having both 
$d\geq 0$ and $c\leq0$ is 
\[
\begin{cases}
    m^2+\frac{3p-2}{p-1}m+\frac{2p}{p-1}\le0,\\
    (3p-4)m+6p\ge0,\\
    m\in[m_1(p)\,,\,m_2(p)].
\end{cases}
\]
This system proves to have a range of admissible values; in particular, it is satisfied by $m=m_1(p)$.
All in all, choosing
$\eta(\f)=(R^2-e^{2\f})^{-\frac{m_1(p)}{2}}$,
which in particular gives
\[\eta\,=\,(R^2-\psi^2)^{-\frac{m_1(p)}{2}},\]
we have that $F$ satisfies the elliptic inequality
\[P^{II}_\f(F)+a\,\langle\nabla F,\nabla\f\rangle_g+ c\, F\ge0,\]
with $c\le0$.
We can thus invoke the Weak Maximum Principle (see Remark \ref{obs max pr}) and deduce that
\[\sup_{\Om} F\le\sup_{\partial\Om}F^+.\]
Note that $\partial\Om=\Gamma_{\rm i} \sqcup\Gamma_{\rm o}$ and in view of \eqref{calc_sigma_in}, it holds 
$${F}_{|\Sigma_{\rm i}}=(R^2-r_i^2)^{-\frac{m_1}{2}}\left(|D u|^p\dot\psi^p-1\right)=(R^2-r_i^2)^{-\frac{m_1}{2}}\left(|\alpha|^p\bigg(\frac{r_i}{R^2-r_i^2}\bigg)^\frac{p}{p-1}-1\right)=0.$$
At the same time, recalling \eqref{w su wr nabla phi} and rewriting $R^2-\Psi^2$ as $\Psi W_R^{\frac{p-1}p}$, we have that at each $y\in\Gamma_{\rm o}$ it holds 
\[\lim_{\Omega\ni x\to y}\,F\,=\lim_{\Omega\ni x\to y}
(R^2-\Psi^2)^{-\frac{m_1}{2}}\left(\frac W{W_R}-1\right)\,=
\lim_{\Omega\ni x\to y}\Psi^{-\frac{m_1}{2}}\left(\frac{W}{W_R^{\frac{(p-1)m_1}{2p}+1}}-W_R^{-\frac{(p-1)m_1}{2p}}\right).\]
Using Lemma \ref{lemma a1 new} and the $\frac1{p^2}$-reverse {\L}ojasiewicz inequality, we have that
\begin{align*}
\lim_{\Omega\ni x\to y}\,\frac{W(x)}{W_R^{\frac{(p-1)m_1}{2p}+1}(x)}&=\lim_{\Omega\ni x\to y}\,\frac{|\nabla u|^p}{\left(\frac{2p}{p-1}(u_{\rm max}-u)\right)^{\frac{(p-1)m_1}{2p}+1}}\\
&\le 
\lim_{\Omega\ni x\to y}c_1\,(u_{\rm max}-u)^{\gamma-\frac{(p-1)m_1}{2p}-1}.
\end{align*}
Since 
$1<p<2$, then $\frac{p}{p-1}< m_1(p)<-2$ and in turn
\[
\frac12
\,<\,
\frac{(p-1)m_1}{2p}+1
\,<\,
\frac1p.
\]
Therefore, since the $\frac{1}{p^2}$-reverse {\L}ojasiewicz inequality allows us to choose above an admissible exponent
$\gamma<1/p$
such that
$\gamma-{\frac{(p-1)m_1}{2p}-1}>0$, we obtain
\[\lim_{x\to y}\frac{W(x)}{W_R^{\,1+\frac{p-1}{2p}m_1(p)}(x)}=0.\]
Moreover, since the exponent $-\frac{(p-1)m_1}{2p}>0$, we have that $W_R^{-\frac{(p-1)m_1}{2p}}\to0$ as we approach $\Gamma_{\rm o}$. We conclude that $F(x)\to0$ as $x\to y\in\Gamma_{\rm o}$. Hence, we have that $F\leq0$ in $\overline\Omega$, which in particular implies $|\N\f|_g\leq1$ in $\Om$. The conclusion of the proof follows the same line as that of Theorem \ref{stime gradiente}.
\endproof

\begin{obs}\label{obs max pr}
{\rm 
For some given functions $a$ and $c$ (possibly depending on $\varphi$), let us consider the linear second-order operator
$$L_\varphi(f) := P^{II}_\varphi(f)+a\,\langle \nabla f,\nabla \varphi\rangle_g + c\,f,\quad\text{in }\Om$$
where $P^{II}_\varphi$ is defined as in \eqref{p_II:def}, $p>1$, and we suppose that $\N\f\not=0$ in $\Om$. If $c\le0$, then \(L_\varphi\) satisfies a weak maximum principle. More precisely, if
\[
L_\varphi(f)\ge 0, \quad\text{ in }\Om
\]
then, under the standard regularity assumptions on the coefficients,
\begin{equation}\label{sup <=}
\sup_{\Omega} f
\le
\sup_{\partial\Omega}f^{+}.    
\end{equation}

The sketchy proof of this result is the following. Let us first address the case $c=0$. One readily sees that if $L_\f(f)>0$, $f$ cannot achieve an interior maximum in $\bar\Om$. Indeed, if $q \in{\rm  int}\,\Om$ is a maximum point for $f$, then 
$$\nabla f(q)=0, \qquad\text{and}\qquad \nabla^2 f(q)\le 0.$$
Note that if $\{x,y\}$ are normal coordinates at $q$, then
\[
P^{II}_\varphi(f)_{|_q}
\,=\,
|\nabla\varphi|^{p-2}\Big\{
\left[(p-1)\varphi_x^2+\varphi_y^2\right]f_{xx}+2(p-2)...\Big\},
\]
and the coefficient matrix is positive definite. Indeed, it is sufficient that
it is nonnegative to deduce that  
\[P^{II}_\varphi(f)(q)\le 0.\]
In turn, this inequality implies that
$L_\f(f)(q)\le 0$, contradicting $L_\f(f)>0$.
Then the maximum of $f$ in $\bar\Om$ is achieved on $\partial\Om$, that is 
$$\sup_{\Om} f\,=\,\sup_{\partial\Om} f.$$
The case where $L_\f(f)\ge0$ in $\Om$, follows by a standard perturbation argument (exploiting the fact that the matrix is indeed definite positive definite at the point where $\N\f\not=0$).
Let us suppose now $c<0$. By considering the subset $\Om^+$ of $\Om$ where $f>0$, we observe that if $L_\f(f)\ge0$ then 
$$P_\f^{II}(f)+a\langle\N f,\N\f\rangle\ge-cf\ge0,\qquad\text{in }\Om^+$$
and hence, reducing to the previous case, the maximum of $f$ in $\bar\Om^+$ must be achieved on $\partial\bar\Om^+$ and hence also on $\partial\Om$. This concludes the proof of \eqref{sup <=}.

Moreover, since under our assumption $1/\lambda$ is locally bounded in $\Om$
then if $c\le0$ the Strong Maximum Principle holds as well (see \cite[Theorem 3.5]{GilTru}).
}
\end{obs}

\subsection{Case 2: external subdomain.}\label{sottosez 5.2}

If $a>b$ and $(\Om,u)$ is a solution to problem \eqref{sist a b alpha beta}, then from Proposition \ref{prop cond pohozaev}, it follows that
$$\alpha\ge0,\qquad \beta<0\qquad \text{and}\qquad a-b>\frac{p-1}{2p}\big(|\beta|^p-\alpha^p\big).$$
So if $a-b\le \frac{p-1}{p}\big(|\beta|^p-\alpha^p\big)$, then using Lemma \ref{lemma a>b} we can associate to $(\Om,u)$ a ring-shaped model solution $(\Om_R,u_R)$, and in particular the comparison parameters are such that $0<\,R\,\le r_i<\,r_o$ and $u_R:[r_i,r_o]\longrightarrow\R$ is defined as
\[u_R(r)\,:=\,a-\int_{r_i}^r\left(\rho-\frac{R^2}{\rho}\right)^\frac{1}{p-1}d\rho.\]
Analogously to before, we introduce the function
$\psi_o:[b,a]\to[r_i,r_o]$ as the inverse of $u_R$ and we get
\begin{equation}\label{psi' e psi'' esterne}
\dot\psi_o=-\left(\frac{\psi_o}{\psi_o^2-R^2}\right)^\frac{1}{p-1}\qquad \text{and}\qquad \ddot\psi_o=-\frac{2}{p-1}(-\dot\psi_o) ^{p+1}+\frac{1}{p-1}\frac{\dot\psi_o^2}{\psi_o}.  
\end{equation}
Now we define on the domain $\Omega$ the \textit{pseudo-radial function}
$$\Psi(x):=\psi_o\left(u(x)\right)$$
and the \textit{pseudo-affine function} 
$$\varphi(x):=-\log\left(\psi_o(u(x))\right).$$
Notice that the choice $-\log$ in the exterior region is only a matter of convention, chosen for notational convenience. Replacing it with $\log$ would not affect the gradient estimate.\\

Following the same argument as before, we rephrase our problem in terms of the new metric $g$.

\begin{prop}\label{pb phi esterno}
Let $a>b$ and $(\Om,u)$ be a solution to problem \eqref{sist a b alpha beta} such that $a-b\le \frac{p-1}{p}\big(|\beta|^p-\alpha^p\big)$.
Then, the metric $g$ and the pseudo-affine function $\f$ satisfy the following system of differential equations over $\Omega$
\begin{equation}\label{sist g e fi outer}
\begin{cases}
    \Delta_p^{(g)}\f=2\psi(-\dot\psi)^{p-1}\Big(|\N\f|^p_g-1\Big)\\
    \mathrm{Ric}_g=-(\Delta_g \f)\,g    
\end{cases}.
\end{equation}    
\end{prop}

Also in the ``decreasing case'' we obtain the same gradient estimates.

\begin{theorem}\label{stime gradiente esterne}
Assume $a>b$, $\alpha\neq0$, and $a-b\le \frac{p-1}{p}\big(|\beta|^p-\alpha^p\big)$.  
Let $(\Om,u)$ be the solution to problem \eqref{sist a b alpha beta} and $R$ the comparison parameter associated to $(\Omega,u)$ by Lemma \ref{lemma a>b}.
Then 
\begin{equation}\label{W<=WR ramo esterno}
    W\le W_R,\quad\text{in }\Omega.
\end{equation}
Moreover, if there exists $x\in\Omega$ such that $W(x)=W_R(x)$ then $(\Om,u)$ is equivalent to $(\Om_R,u_R)$.    
\end{theorem}

\proof
Building on the reasoning presented in the proof of Theorem \ref{stime gradiente beta=0} using, in this case, the conformal reformulation of the problem given by Proposition \ref{pb phi esterno}, we obtain the following elliptic inequality for the quantity $F=\eta(\f)\,w$ 
\begin{align}\notag
P^{II}_\f(F)&\ge\left[(p-1)\frac{\eta''}{\eta}-2(p-1)\left(\frac{\eta'}{\eta}\right)^2-\left(p-2+2p\psi(-\dot\psi)^{p-1}\right)\frac{\eta'}{\eta}+p\psi^2\frac{6R^2-2\psi^2}{(\psi^2-R^2)^2}\right]|\nabla\f|_g^p F\\ \notag
&+\left[p-2+2p\psi(-\dot\psi)^{p-1}+2(p-1)\frac{\eta'}{\eta} \right]|\nabla\f|_g^{p-2}\langle\nabla F,\nabla\f\rangle_g\\ \notag
&+\frac{2(2-p)}{\eta}\left[2\psi(-\dot\psi)^{p-1}+\frac{3p-2}{2p}\frac{\eta'}{\eta}\right]\,F\, \frac{\langle\nabla F,\nabla\f\rangle_g}{|\nabla\f|_g^2}+\frac{1}{\eta}\left[p-2+ \frac{(p-2)^2}{2p}\right]\frac{\langle\nabla F,\nabla\f\rangle_g^2}{|\nabla\f|_g^4}\\ \label{disug P^II F est}
&+\frac{1}{\eta}\left[2\psi(-\dot\psi)^{p-1}\left((2p-3)\frac{\eta'}{\eta}+p\psi(-\dot\psi)^{p-1}\right)+(p-2)\frac{3p-2}{2p}\left(\frac{\eta'}{\eta}\right)^2\right] F^2.
\end{align}    
We point out that in this case we can not argue as in the proof of Theorem \ref{stime gradiente} (where $F=w$), since we would obtain a coefficient $c$ not having definite sign. The rest of the proof is analogous to the proof of Theorem \ref{stime gradiente beta=0}. In this case, we set 
\begin{equation*}
\frac{\eta'(\f)}{\eta(\f)}=m\psi(-\dot\psi)^{p-1} 
\end{equation*}
and we rewrite the above inequality as follows
\begin{equation}\label{disug ell Feta nel caso esterno}
P^{II}_\f(F)+a\,\langle\nabla F,\nabla\f\rangle_g+ c\, F\ge d F^2.    
\end{equation}
A sufficient condition for having both 
$d\geq 0$ and $c\leq0$ is 
\[\begin{cases}
    m^2+\frac{3p-2}{p-1}m+\frac{2p}{p-1}\le0,\\
    (3p-4)m+6p\ge0,\\
    m\in[m_1(p)\,,\,m_2(p)]
\end{cases}\]
where 
\begin{equation}
\label{eq:m_1}
m_1(p)=\frac{-2p}{2p-3+\sqrt{p^2-4p+5}},
\qquad m_2(p)=\frac{-2p}{2p-3-\sqrt{p^2-4p+5}}. 
\end{equation}
We choose $m=m_1(p)$, so that
$$\eta(\f)=\big(e^{-2\f}-R^2\big)^{-\frac{m_1(p)}{2}}.$$
We conclude using the maximum principle argument and noticing that in this case $F$ is well-defined on both the two boundary components, in particular, $F=0$ on $\partial\Om=\Gamma_{\rm i}\sqcup\Gamma_{\rm o}$.
\endproof

\begin{theorem}\label{stime gradiente alfa=0}
Assume $a>b$, $\alpha=0$, and $a-b\le \frac{p-1}{p}|\beta|^p$.  
Let $(\Om,u)$ be the solution to problem \eqref{sist a b alpha beta}
and $R$ the comparison parameter given by Lemma \ref{lemma a>b}.
Assume further that $u$ satisfies a $\frac{1}{p^2}$-reverse {\L}ojasiewicz inequality at $\Gamma_{\rm i}$, according to Definition \ref{RRLI}. 
Then
\begin{equation*}
    W\le W_R,\quad\text{in }\Omega.
\end{equation*}
Moreover, if there exists $x\in\Omega$ such that $W(x)=W_R(x)$, then $(\Om,u)$ coincides with $(\Om_R,u_R)$, up to translations.  
\end{theorem}

\proof
We argue as in the proof of Theorem \ref{stime gradiente beta=0}, exploiting the elliptic inequality \eqref{disug P^II F est}.
\endproof

\subsection{Auxiliary results.}

We collect here some tools that we have used in order to prove the gradient estimates in the case when one of the two boundary components of $\Omega$ is a curve of maximum points (see Theorem \ref{stime gradiente beta=0} and Theorem \ref{stime gradiente alfa=0}).

\begin{lemma}\label{lemma a1 new}
Let $(\Om,u)$ be a solution to problem \eqref{sist interno} (analogously for a solution to \eqref{sist esterno}). Let $R$ be the comparison parameter associated to $(\Omega,u)$.
Then, for every $y\in\Sigma$ it holds
$$\lim_{x\to y,\,\,x\in\Omega}\,\frac{W_R}{u_{\rm max}-u}=\frac{2p}{p-1}.$$  
The same result holds when we consider a solution to \eqref{sist esterno}, in this case, we have to assume that $u_{\rm max}< \frac{p-1}{p}|\beta|^p$. 
\end{lemma}

\proof
Recall that 
$$W_R=\left(\frac{R^2-\Psi^2}{\Psi}\right)^\frac{p}{p-1} \qquad\text{and}\qquad u(x)=\int_{r_i}^{\Psi(x)} \left(\frac{R^2}{\rho}-\rho\right)^\frac{1}{p-1}d\rho.$$
We have 
$$\frac{W_R}{u_{\rm max}-u}=\left(\frac{R^2-\Psi^2}{\Psi}\right)^\frac{p}{p-1}\frac{1}{\int\limits_{\Psi(x)}^R \left(\frac{R^2}{\rho}-\rho\right)^\frac{1}{p-1}d\rho},$$
hence using L'Hôpital's rule
$$\lim_{x\to y,\,\,x\in\Omega}\,\frac{W_R}{u_{\rm max}-u}=
\lim_{\xi\to R}\,\frac{\xi^{-\frac{p}{p-1}}(R^2-\xi^2)^\frac{p}{p-1}}{\int\limits_{\xi}^R \left(\frac{R^2}{\rho}-\rho\right)^\frac{1}{p-1}d\rho}
=R^{-\frac{p}{p-1}} \lim_{\xi\to R}\,\frac{-\frac{2p}{p-1}\,\xi\,(R^2-\xi^2)^{\frac{p}{p-1}-1} }{-\left(\frac{R^2}{\xi}-\xi\right)^\frac{1}{p-1}}=\frac{2p}{p-1}.$$
\endproof

Finally, the condition introduced in the following definition will be useful in dealing with solutions near their maximum points, for sufficiently small $p\in(1,2)$.

%-----------------------------------------------------------------------

\section{Characterization of $p$-Serrin solutions}\label{sez6}

In this section, let us consider problem \eqref{nostro problema} in the more general case where we do not require that $\Om$ is a ring-shaped domain. 
The above-mentioned result by Garofalo and Lewis \cite{GarofaloLewis} (in the spirit of the classical result due to Serrin \cite{Serrin}) states that, if the pair $(\Om,u)$ solves \eqref{nostro problema} and the normal derivative is constant on $\partial\Om$, then necessarily $\Om$ is a ball and $u$ is rotationally symmetric. 

Next, we prove a characterization of $p$-Serrin solutions in terms of the quantity $\max_{\partial\Om}\frac{|\nabla u|^p}{u_{\rm max}}$. 

\begin{theorem}\label{teo p serrin omega}
Let $\Omega$ be a bounded planar domain with smooth boundary, and let $u$ be the solution to problem \eqref{nostro problema} in $\Omega$. 
Assume that 
\begin{equation}\label{stima per p serrin}
\max_{\partial\Om}\frac{|\nabla u|^p}{u_{\rm max}}\le\frac{p}{p-1},
\end{equation} 
and that Assumption {\rm \blu{I}} holds. Assume further that $u$ satisfies a
$\frac{2-p}{p^2}$-reverse {\L}ojasiewicz inequality at ${\rm MAX}(u)$, according to Definition \ref{RRLI}. 
Then, $u$ is rotationally symmetric and $\Om$ is a disk.
\end{theorem}

\proof
The proof is based on a maximum principle argument. 
First, consider the function 
$$F:=\Psi^{-m_1(p)-\frac{p}{p-1}}(W-W_0),$$
where 
$$W_0:=|\nabla u_0|^p\circ \Psi=\Psi^\frac{p}{p-1},$$
$m_1(p)$ is defined in \eqref{eq:m_1} and $\Psi(x):=\psi_o\big(u(x)\big)$ is defined as in Subsection \ref{sottosez 5.2} assuming $R=0$, in this case $\psi_o:[0,u_{\rm max}]\to [0,r_o]$ where $r_o=\left(\frac{p}{p-1} u_{\rm max}\right)^\frac{p-1}{p}$.  Notice that in what follows also $w$ and $\varphi$ are defined as in Subsection \ref{sottosez 5.2}.
The function $F$ satisfies the following elliptic inequality 
$$P^{II}_\f(F)+a\,\langle\nabla F,\nabla\f\rangle_g+cF\ge0,\qquad\text{in }\,\,U:=\Om\setminus {\rm MAX}(u)$$
To obtain it, we argue as in the proof of Theorem \ref{stime gradiente esterne} setting $R=0$.
In this case it holds $\dot\psi=-\psi^{-\frac{1}{p-1}}$, hence $\frac{\eta'(\f)}{\eta(\f)}=m\in \R$. In particular, we can fix $m=m_1(p)$, so that $F=\Psi^{-m_1} w=\Psi^{-m_1(p)-\frac{p}{p-1}}(W-W_0)$, and moreover $c\le0$ and $d=0$ in the inequality \eqref{disug ell Feta nel caso esterno}. 
Since
$$W-W_0=|\nabla u|^p+\frac{p}{p-1}u-\frac{p}{p-1}u_{\rm max}$$
by the weak maximum principle, we get
$$\max_{U} F=\max_{\partial U}F=\max\left\{\max_{\partial \Om}F\,,\, \max_{{\rm MAX}(u)}F\right\}.$$
Notice that since $u=0$ on $\partial\Om$
$$\max_{\partial\Om}F=\max_{\partial \Om}\Big[\Psi^{-m_1(p)-\frac{p}{p-1}} \Big( |\nabla u|^p-\frac{p}{p-1}u_{\rm max}\Big)\Big]\le 0$$
where in the last inequality we exploited the fact that $\Psi\ge0$ and assumption \eqref{stima per p serrin}.
On the other hand, 
$$F=\frac{W}{W_0^{1+\frac{p-1}{p}m_1}}-W_0^{-\frac{p-1}{p}m_1}$$
we want to study the limit of this quantity as it approaches ${\rm MAX}(u)$. 
For any $x\in\Om$ it holds $\frac{W_0}{u_{\rm max}-u}=\frac{\Psi^\frac{p}{p-1}}{\frac{p-1}{p}\Psi^\frac{p}{p-1}}=\frac{p}{p-1}$, 
so we get, for every $y\in{\rm MAX}(u)$,
\begin{equation}\label{lim W su W0 usa loj}
\lim_{\Omega\ni x\to y}\,\frac{W(x)}{W_0^{1+\frac{p-1}{p}m_1}(x)}=\lim_{\Omega\ni x\to y}\,\frac{|\nabla u|^p}{\left(\frac{p}{p-1}(u_{\rm max}-u)\right)^{1+\frac{p-1}{p}m_1}}
\leq 
\lim_{\Omega\ni x\to y}c\,(u_{\rm max}-u)^{\gamma-\frac{(p-1)m_1}{p}-1},
\end{equation}
where in the inequality
we have used the 
$\frac{2-p}{p^2}$-reverse {\L}ojasiewicz inequality.
It is easy to check that \(m_1(p)<-2\) for every $p\in(1,2)$, so that 
\[
\frac{p-1}{p}m_1(p)
+1
<\frac{2-p}{p}.\]
Hence, by the
$\frac{2-p}{p^2}$-reverse {\L}ojasiewicz inequality, we may choose above an admissible exponent 
$\gamma<\frac{2-p}p$ such that
$\gamma
-\frac{(p-1)m_1(p)}{p}-1>0$.
Consequently, 
\[
\lim_{x\to y}\frac{W(x)}{W_0^{\,1+\frac{p-1}{p}m_1(p)}(x)}=0.
\]
We have thus concluded that $F=0$ at ${\rm MAX}(u)$. In turn, 
\[
F=\Psi^{-m_1(p)-\frac{p}{p-1}}\bigg(|\nabla u|^p+\frac{p}{p-1}u-\frac{p}{p-1}u_{\rm max}\bigg)\equiv0 \quad\text{ in }\,U,
\]
and 
hence 
$$|\nabla u|^p+\frac{p}{p-1}u-\frac{p}{p-1}u_{\rm max}\equiv0 \quad\text{ in }\Om.$$
The fact that $W-W_0\equiv0$ in $\Om$ implies that $|\nabla\f|_g\equiv1$ in $\Om$, that is $w\equiv0$ in $\Om$. The remainder of the proof proceeds along the same lines as the rigidity argument in the proof of Theorem \ref{stime gradiente}. 
\endproof

We next define the auxiliary function $U$, whose monotonicity property will play a key role in the proof of Theorem \ref{teo2.1NEW}.

\begin{prop}\label{prop2.3new}
Let $(\Om,u)$ be a solution to \eqref{nostro problema} where $\Om$ is a bounded domain with smooth boundary. Assume that ${\rm Crit}(u)\setminus{\rm MAX}(u)=\varnothing$ and if $N$ connected component of $\Om\setminus{\rm MAX}(u)$ it holds
\[
\max_{\partial\Om\cap\overline N}\frac{|\nabla u|^p}{u_{\rm max}}\,\leq\,\frac p{p-1}.
\]
Then,
\begin{itemize}
\item[i)] the function $U:[0,u_\mathrm{max})\to\R$ defined as 
$$U(t):=\frac{1}{(u_{\rm max}-t)^{2\big(\frac{p-1}p\big)}}\int\limits_{\{u=t\}\cap \overline{N}} |\nabla u|^{p-1}d\sigma.$$
is non-increasing. 
\item[ii)] If $u$ satisfies a {\L}ojasiewicz-type inequality at ${\rm MAX}(u)$ (see Definition \ref{ass_Loja}), and $\mathcal{H}^{1}(\mathrm{MAX}(u)\cap\overline{N})>0$ then 
\begin{equation}\label{limU}
    \lim_{t\to u^-_\mathrm{max}}U(t)=+\infty.
\end{equation}
\end{itemize}
\end{prop}

\proof 
The function $U$ is well-defined, because $|\nabla u|$ is 
bounded in $\overline\Omega$ and the level sets of $u$ have locally finite $\mathcal{H}^1$-measure
An application of the maximum principle yields the following  estimate 
$$|\nabla u|^p+\frac{p}{p-1}u-\frac{p}{p-1}u_{\rm max}\le 0,\quad \text{ on }N$$
(this follows by applying the maximum principle on $N_\e:=N\cap\{u\le u_{\rm max}-\e\}$ for $\e\in\R$ arbitrary small as in the proof of Theorem \ref{teo p serrin omega} and then letting $\e\to0^+$, finally recalling that $\Psi>0$ on $N$).
We have that
$$\mathrm{div}\left(\frac{|\nabla u|^{p-2}\nabla u}{(u_{\rm max}-u)^\frac{2p-2}{p}}\right)=\frac{2p-2}{p(u_{\rm max}-u)^\frac{3p-2}{p}}\Big(|\nabla u|^p+\frac{p}{p-1}u-\frac{p}{p-1}u_{\rm max}\Big)\le0$$
Integrating by parts on $\{t_1\le u\le t_2\}\cap\overline{N}$ together with the definition of $U$ gives
$$U(t_2)-U(t_1)\le0$$
which gives the desired monotonicity and therefore point $i)$.

From the {\L}ojasiwicz Inequality and the compactness of $\mathrm{MAX}(u)$, we obtain that for $t$ sufficiently close to $u_\mathrm{max}$
\begin{equation}\label{U>=}
U(t)\ge\frac{C}{(u_\mathrm{max}-t)^{(p-1)\left(\frac{2}{p}-\gamma\right)}}\,\Big|\{u=t\}\cap\overline{N}\Big|,  
\end{equation}
where $C>0$ and $\gamma<\frac{2}{p}$. 
Finally, since $\mathcal{H}^1(\mathrm{MAX}(u)\cap\overline{N})>0$ it holds 
\begin{equation}\label{limsup>0}
\limsup\limits_{t\to u^-_\mathrm{max}}\big|\{u=t\}\cap \overline{N}\big|>0,
\end{equation}
and the conclusion follows.
\endproof

We are now in the position to prove Theorem \ref{teo2.1NEW}, which is a localized version of Theorem \ref{teo p serrin omega}.

\proof[Proof of Theorem \ref{teo2.1NEW}]
We are under the assumptions of 
Proposition \ref{prop2.3new}, so from the monotonicity of $U$, we get that
$$\lim\limits_{t\to u^-_\mathrm{max}}U(t)\le U(0)=\frac{1}{ {u_\mathrm{max}}^\frac{2p-2}{p}}\int\limits _{\partial\Omega\cap\overline{N}} |\nabla u|^{p-1}d\sigma\le\frac{1}{ {u_\mathrm{max}}^{\frac{p-1}{p}}} \left(\frac{p}{p-1}\right)^\frac{p-1}{p}\big|\partial\Omega\cap\overline{N}\big|<+\infty.$$
It follows that $\mathcal{H}^1(\mathrm{MAX}(u)\cap\overline{N})=0$, otherwise $\lim_{t\to u^-_\mathrm{max}}U(t)=+\infty$ which would lead to a contradiction.
The fact that $\mathcal{H}^1(\mathrm{MAX}(u)\cap \overline{N})=0$ implies that $\Omega\setminus\mathrm{MAX}(u)$ is connected, and hence $\partial\Omega\cap \overline{N}=\partial\Omega$ this fact coupled with Theorem \ref{teo p serrin omega}, allows us to conclude.
\endproof

%-----------------------------------------------------------------

\section{Length bounds for $\Gamma_{\rm i}$ and $\Gamma_{\rm o}$}\label{sez7}

The gradient estimates obtained in the previous section play a key role in establishing bounds for the lengths of the inner and outer boundary components of the domain $\Omega$. Since the method used to derive these bounds fails when $|\nabla u|=0$, we separately provide
the length bounds for the curve of maximum pints $\Sigma$.

Notice that from Proposition \ref{prop cond pohozaev}, if $u$ solves \eqref{sist a b alpha beta} and $a<b$ then we have  $\alpha<0$ and $\beta\ge0$, hence $|\nabla u|$ cannot vanish on $\Gamma_{\rm i}$. Analogously, if $u$ solves \eqref{sist a b alpha beta} and $a>b$, then $\alpha\ge0$ and $\beta<0$. So, we will treat separately the case $\beta=0$ when $a<b$, and $\alpha=0$ when $a>b$ (see Proposition \ref{prop bound derivata norm nulla} below). These estimates can be fruitfully exploited to obtain our main rigidity results.

\begin{prop}\label{prop length bnds Gamma_i e out}
Let $(\Om,u)$ be a solution to problem \eqref{sist a b alpha beta} such that $a<u<b$ and $\beta\not=0$.
Then it holds 
\begin{equation}\label{stime lungh}
|\Gamma_{\rm i}|\le2\pi\,r_i\qquad\text{and}\qquad |\Gamma_{\rm o}| \ge2\pi\,r_o,
\end{equation}
where $r_i$ and $r_o$ are the two values given by Lemma \ref{lemma a<b}.

The same estimates also hold when $b<u<a$. In this case we assume that $\alpha\not=0$ and $a-b\le\left(\frac{p-1}{p}\right)(|\beta|^p-\alpha^p)$ (using Lemma \ref{lemma a>b} to determine $r_i$ and $r_o$).
\end{prop}

\proof
We start proving the length bounds in the case when $a<u<b$.
Let us consider $q\in\Gamma_{\rm i}$. We compute the Taylor expansion of both $W$ and $W_R$ along the half-line $\gamma(s)=q+\frac{\nabla u}{|\nabla u|}(q) s$
$$W\left(\gamma(s)\right)=W(q)+\langle\,\nabla W(q)\,,\,\frac{\nabla u}{|\nabla u|}(q)\,\rangle s+o(s).$$
Recall now that the curvature $k$ of $\Gamma_{\rm i}$ with respect to the exterior unit normal $-\frac{\nabla u}{|\nabla u|}$ is 
$$k(q)=\frac{\nabla^2 u\big|_q(\nabla u, \nabla u)}{|\nabla u|^3(q)}-\frac{\Delta u(q)}{|\nabla u|(q)}\,=\,\frac{2}{|\nabla u|^{p-1}(q)}+(p-1)\frac{\nabla^2 u\big|_q(\nabla u,\nabla u)}{|\nabla u|^3}$$
where in the last equality we exploited the fact that $-\Delta_p u=2$, so isolating the hessian we get 
$$\nabla^2 u\big|_q(\nabla u,\nabla u)=\left(k(q)-\frac{2}{|\nabla u|^{p-1}(q)}\right)\frac{|\nabla u|^3(q)}{p-1}.$$ 
Using this identity in the expansion for $W$, we obtain 
\begin{align}\notag
W\left(\gamma(s)\right)&=W(q)+\frac{p}{p-1}|\nabla u|^p\left(k(q)-\frac{2}{|\nabla u|^{p-1}}\right)\,s+o(s)\\ 
&=W(q)-\frac{p}{p-1}W(q)\left(\frac{2}{W(q)^\frac{p-1}{p}}-k(q)\right)\,s+o(s). \label{taylor W}
\end{align}
On the other hand, since $\nabla W_R=-\frac{p}{p-1}\left(\frac{\Psi^2+R^2}{\Psi^2}\right)\nabla u$ (notice that we are dealing with the case when $a<u<b$ so 
$\Psi(x)=\psi_i\left(u(x)\right)$ as defined above) we obtain
\begin{equation}\label{taylor Wc}
W_R(\gamma(s))=W_R(q)-\frac{p}{p-1}\left(\frac{\Psi^2+R^2}{\Psi^2}\right)|\nabla u|\,(q)\,s+o(s).  
\end{equation}
Since $\partial_\nu u=\alpha$ on $\Gamma_{\rm i}$, it holds $W(q)=|\nabla u|^p(q)=|\alpha|^p$, and since $\Psi(q)=r_i$ then $W_R(q)=\left(\frac{R^2-r_i^2}{r_i}\right)^\frac{p}{p-1}=|\alpha|^p$ (where in the last equality we used that $R^2-r_i^2=|\alpha|^{p-1}r_i$). We are now in the position to compare the two expressions \eqref{taylor W} and \eqref{taylor Wc}. By exploiting the gradient estimates given by Theorem \ref{stime gradiente}, we finally get
$$k(q)\le -\frac{1}{\Psi(q)}=-\frac{1}{r_i}$$
and integrating on $\Gamma_{\rm i}$
$$-2\pi\le-\frac{|\Gamma_{\rm i}|}{r_i}.$$

We can argue similarly for $\Gamma_{\rm o}$. Let us consider $q\in\Gamma_{\rm o}$, since in this case the exterior unit normal is $\frac{\nabla u}{|\nabla u|}$ we take $\gamma(s)=q-\frac{\nabla u}{|\nabla u|}(q) s$ and we have 
$$k=-\frac{\nabla^2 u(\nabla u, \nabla u)}{|\nabla u|^3}+\frac{\Delta u}{|\nabla u|}\,=\,-\frac{2}{|\nabla u|^{p-1}}-(p-1)\frac{\nabla^2 u(\nabla u,\nabla u)}{|\nabla u|^3}.$$
Using the gradient estimate and the fact that $W(q)=W_R(q)=\beta^p$ we obtain 
$$k(q)\le\frac{1}{r_o}$$
and hence 
$$2\pi\le\frac{|\Gamma_{\rm o}|}{r_o}.$$

We can recover the same argument as before also in the case $b<u<a$,  recall that in this case $\Psi(x)=\psi_o\left(u(x)\right)$ and 
$W_R=\left(\frac{\Psi^2-R^2}{\Psi^2}\right)^\frac{p}{p-1}$.
\endproof

Observe that when $|\nabla u|=0$ on a boundary component, namely, it is a curve of maximum points, some of the terms in the previous proof are not well-defined. So a different approach is needed to prove length bounds for the outer boundary component $\Sigma$ when considering a solution to system \eqref{sist interno}, and for the inner boundary component (still denoted by $\Sigma$) when considering a solution to system \eqref{sist esterno}. Indeed, we will take advantage of some asymptotic expansions for $|\N u|^{p-1}$ and $|\dot u_R|^{p-1}$ when approaching the set of maximum points $\Sigma$. Notice that in these cases, further (but reasonable) assumptions are needed.

%--------------------------------------------------------------------

\section{Length bounds for $\Sigma$}\label{sez8}

In this section, we obtain the corresponding length bounds for the curve of maximum points $\Sigma$. The proof of this result is based on the asymptotic expansions for $V=|\nabla u|^{p-1}$ and $V_R:=|\dot u_R|^{p-1}$ about $\Sigma$, which we can compare using the previous gradient estimates. To keep the presentation cleaner, we postpone to Subsection \ref{subsec stime}, the details about the expansions for $V$ and $V_R$ at $\Sigma$.\\

Let $(\Om,u)$ be a solution to problem \eqref{sist interno} (the case where $(\Om,u)$ solves system \eqref{sist esterno} is analogous). Recall that in this case $\partial\Om=\Sigma_{\rm i}\sqcup\Sigma$. Let $t$ be the distance function from $\Sigma$, that is, given $x\in\Om$, $t(x):=\mathrm{dist}(x,\Sigma)$, so we can define the corresponding local coordinates $(t,\theta)$. We introduce the function $v_R$, that is the model solution $u_R$ seen as a function of the distance $ t$
$$v_R( t):=u_R(R- t)=\int_{r_i}^{R- t}\left(\frac{R^2}{\rho}-\rho\right)^\frac{1}{p-1}d\rho.$$
and also the functions
$$V_R(x):=|\dot u_R|^{p-1}_{|\Psi(u(x))}\qquad\text{and}\qquad V=|\nabla u|^{p-1},$$
in what follows, with a slight abuse of notation, we denote by $V_R( t,\theta)$ the quantity $V_R(x( t,\theta))=|\dot u_R|^{p-1} _{|\Psi(u( t,\theta))}$. 

We obtain the following length bounds for the curve $\Sigma$.

\begin{prop}\label{prop bound derivata norm nulla}
Let $(\Om,u)$ be a solution to problem \eqref{sist interno}. Let us assume that the asymptotic expansion \eqref{ass_esp} holds for $u$ at $\Sigma$. 
Then it holds 
$$|\Sigma|\ge 2\pi R.$$
Analogously, if $(\Om,u)$ solves problem \eqref{sist esterno}, assuming further that $u_{\rm max}\le \frac{p-1}{p}|\beta|^p$, we have that
$$|\Sigma|\le 2\pi R.$$
Where the parameter $R$ is given by Lemma \ref{lemma a<b} and Lemma \ref{lemma a>b}, respectively.
\end{prop}

\proof
Let $(\Om,u)$ a solution to problem \eqref{sist interno}. From the gradient estimate given by Theorem \ref{stime gradiente beta=0}, we have that $V\big(x(t,\theta)\big)\le V_R\big(x(t,\theta)\big)$). So, exploiting the Taylor expansions for $V$ and $V_R$ as $t\to0^+$  given by Proposition \ref{lemma7.1} and Proposition \ref{prop stima VR} (both below), we get
$$k(\theta)\le\frac{1}{R}$$
where $k(\theta)$ is the curvature of $\Sigma$ at the point $(0,\theta)$ computed with respect to the exterior unit normal to $\Om$. The desired length bound follows immediately. 

When considering problem \eqref{sist esterno}, 
by proceeding with the same argument (use Proposition \ref{prop stima VR esterna}), we have that the curvature of $\Sigma$ is less or equal to $-\frac{1}{R}$. Since in this case the curve $\Sigma$ is oriented clockwise, integrating over $\Sigma$, we get $-2\pi\le -\frac{|\Sigma|}{R}$.
\endproof

\subsection{Asymptotic estimates at $\Sigma$}\label{subsec stime}

We explicitly deal with the case where $u$ is a solution to problem \eqref{sist interno} (the case where $u$ solves system \eqref{sist esterno} is analogous, so we provide the corresponding results, but we omit the details). 
The vector field
\begin{equation}\label{vfX}
X:=|\N u|^{p-1}\left(\frac{\nabla u}{|\nabla u|}+\nabla  t\right),
\end{equation}
defined over $\Omega\setminus\Sigma$,
plays a crucial role in the subsequent analysis. We will denote by
$$X^t:=\langle X,\frac{\partial}{\partial t}\rangle\qquad\text{ and }\qquad X^\theta=\langle X,\frac{\partial}{\partial\theta}\rangle.$$

We obtain a Taylor expansion for $V$ at $\Sigma$, namely as $t\to0^+$. 

\begin{prop}\label{lemma7.1}
Let us assume that the asymptotic expansion \eqref{ass_esp} holds for $u$ at $\Sigma$. Then
\begin{equation}\label{eq:coeff_expl}
a(\theta)=-\left(\frac{p-1}{p}\right)2^\frac{1}{p-1},\qquad\quad b(\theta)
=-\left(\frac1{2p-1}\right)2^\frac{2-p}{p-1}\,\,k(\theta),
\end{equation}
and as $t\to 0^+$
\begin{equation}\label{ff5}
V_{|( t,\theta)}=2 t+ k(\theta)\, t^2+o( t^2).
\end{equation}
\end{prop}

\proof
Expansion \eqref{ass_esp} easily gives in particular that
\begin{align*}
\frac{\pa u}{\pa t}&=\frac{p}{p-1}\, a(\theta)\,\,t^{\frac1{p-1}}
+o\left(t^{\frac1{p-1}}\right),\\
\frac{\pa u}{\pa\theta}
&=a'(\theta)t^{\frac p{p-1}}+b'(\theta)\,t^{\frac{p}{p-1}+1}+o \left(t^{\frac p{p-1}+1}\right). 
\end{align*}
From this, we get
\[|\nabla u|^2=\left(\frac{\pa u}{\pa t}\right)^2+\frac1{\varphi^2} \left(\frac{\pa u}{\pa\theta}\right)^2
=a(\theta)^2\left(\frac p{p-1}\right)^2 t^{\frac2{p-1}}\,\left[1+\frac{2(2p-1)}{p}\frac{b}{a}t+o(t)\right],
\]
recall that we denote by $\Sigma_t(\cdot):=x( t,\cdot)$ and $\varphi( t,\theta):=|\dot\Sigma_ t(\theta)|$, so the metric can be written as $g={\rm d} t^2+\varphi^2( t,\theta){\rm d}\theta^2$. 
This in turn implies that
$$V=\left(|\nabla u|^2\right)^\frac{p-1}{2}=|a(\theta)|^{p-1}\left(\frac p{p-1}\right)^{p-1}\,t\,\left[1+\frac{(p-1)(2p-1)}{p}\frac{b}{a}\,t\,+o(t)\right]$$
and hence
\begin{equation}\label{basic_X^t_0}
\frac{\partial V}{\partial t}\bigg|_{(0,\theta)}=|a(\theta)|^{p-1}\left(\frac p{p-1}\right)^{p-1},\qquad X^t = |\nabla u|^{p-2}\frac{\pa u}{\pa t}+V= o(t).
\end{equation}
We have that
$${\rm div}_g X=\frac{\partial X^t}{\partial t }+\frac{1}{\varphi}X^t \frac{\partial \varphi}{\partial t }+\frac{\partial X^\theta}{\partial \theta }+\frac{1}{\varphi} X^\theta \frac{\partial \varphi}{\partial\theta}.$$
Since $|\nabla u|^{p-1}\to 0$ as $ t\to0^+$, and $|\frac{\partial}{\partial t}\big|_{ t=0}\varphi ( t,\theta)|=|k(\theta)|$, we get that 
$$\lim_{ t\to0^+} {\rm div} X_{|( t,\theta)}= \lim_{ t\to0^+}\frac{\partial X^ t}{\partial  t} \bigg|_{( t,\theta)},$$
and from the second identity in \eqref{basic_X^t_0} we obtain
\begin{equation}
\label{Mess_0}
\lim_{t\to0^+}{\rm div X}_{|(t,\theta)}=0.
\end{equation}
Now, we observe that 
$$\begin{aligned}
-2=\Delta_p u&=\mathrm{div}\left(|\nabla u|^{p-1}\frac{\nabla u}{|\nabla u|}\right)=\langle\nabla V,\frac{\nabla u}{|\nabla u|}\rangle + V\mathrm{div}\left(\frac{\nabla u}{|\nabla u|}\right)\\ 
&=-\frac{\partial V}{\partial t}- V\Delta t+\mathrm{div} \left[V\left(\frac{\nabla u}{|\nabla u|}+\nabla t\right)\right]\\
&=-\frac{\partial V}{\partial t}- V\Delta t+\mathrm{div} X,
\end{aligned}$$
so that
\begin{equation}\label{ff2}
-2=-\frac{\partial V}{\partial t}\bigg|_{(t,\theta)}+ V_{|(t,\theta)} k(t,\theta)+\mathrm{div} X_{|( t,\theta)}
\end{equation}
From \eqref{Mess_0} we therefore get
\begin{equation}\label{Mess_1}  
\frac{\partial V}{\partial t}\bigg|_{(0,\theta)}=2.
\end{equation}
This identity coupled with the first identity in \eqref{basic_X^t_0} and with the fact that $a(\theta)<0$ gives the explicit expression for the coefficient $a(\theta)$ stated in \eqref{eq:coeff_expl}.
We now have
\begin{align*}
\frac{\pa u}{\pa t}
&=-2^{\frac 1{p-1}}
t^{\frac1{p-1}}+b(\theta)
\left(\frac{2p-1}{p-1}\right)t^{\frac p{p-1}}
+o\left(t^{\frac p{p-1}}\right),\\
\frac{\pa u}{\pa\theta}
&=b'(\theta)t^{\frac{2p-1}{p-1}}+o\left(t^{\frac{2p-1}{p-1}}\right).
\end{align*}
Hence,
\[|\nabla u|^2=(2t)^{\frac2{p-1}}\left[
1-2^{\frac{p-2}{p-1}}b(\theta)\left(\frac{2p-1}{p-1}\right)t+o(t)\right],\]
and in turn
\[V=2t\left[1-2^{-\frac1{p-1}}b(\theta)\left(2p-1\right)t+o(t)\right].\]
Therefore, we obtain, after a straightforward computation, that
\begin{equation}\label{eq:basic_X^t}
X^t=|\nabla u|^{p-2}\frac{\pa u}{\pa t}+V=o(t^2).
\end{equation}
from \eqref{eq:basic_X^t} we get
\begin{equation}
\label{Mess_2}
\lim_{t\to0^+}\frac{\partial}{\partial t}{\rm div X}_{|(t,\theta)}=0.
\end{equation}
By differentiating \eqref{ff2} with respect to $ t$, one gets
$$0=-\frac{\partial^2 V}{\partial^2  t}+ k(t,\theta)\frac{\partial V}{\partial t}+ V_{|(t,\theta)}\frac{\partial k(t,\theta)}{\partial t}
+\frac{\partial}{\partial  t}\mathrm{div} X,$$
and using \eqref{Mess_1}, \eqref{Mess_2}, and that $V(0,\theta)=0$, we get
\begin{equation}\label{ff4}
\frac{\partial^2 V}{\partial t^2}\bigg|_{{(0,\theta)}}=2k(\theta).   
\end{equation}
Finally, by using relations \eqref{Mess_1} and \eqref{ff4}, we get the desired expansion for $V$.
We finally combine the expansion for $V$, the identity $V=\left(\left(\frac{\partial u}{\partial t}\right)^2+\frac{1}{\varphi^2}\left(\frac{\partial u}{\partial \theta}\right)^2\right)^\frac{p-1}{2}$, and the known value for $a(\theta)$, to obtain that
\begin{equation}\label{ff16}
\frac{\partial u}{\partial t}=-2^\frac{1}{p-1}t^\frac{1}{p-1}-\frac{1}{p-1} 2^{\frac{1}{p-1}-1}k(\theta)t^{\frac{1}{p-1}+1}+o(t).   
\end{equation}
Using l'Hopital's rule and \eqref{ff16}, we have
$$\lim_{t\to0^+}\frac{u_{\rm max}-u-\frac{p-1}{p}2^\frac{1}{p-1} t^\frac{p}{p-1}}{t^{\frac{p}{p-1}+1}}=\frac{1}{2p-1}2^\frac{2-p}{p-1}k(\theta),$$
which gives the explicit expression for $b(\theta)$.
\endproof

The following lemma provides the asymptotic expansion of the function $\psi(u)$ around $u=u_{\rm max}$. Recall that $\psi$ is the pseudo-radial function defined in Section \ref{sez stime grad}.

\begin{lemma}\label{espansione phi}
If the asymptotic expansion \eqref{ass_esp} holds for $u$ at $\Sigma$, then the function $\psi$ is such that
\begin{align*}
\psi(u)-R=&-2^{-\frac{1}{p}}\left(\frac{p}{p-1}\right)^\frac{p-1}{p}(u_{\rm max}-u)^\frac{p-1}{p}\\
&+2^{-1-\frac{2}{p}}\frac{1}{2p-1} \left(\frac{p}{p-1}\right)^{2\big(\frac{p-1}{p}\big)}\frac{1}{R}(u_{\rm max}-u)^{2\big(\frac{p-1}{p}\big)}+o\bigg((u_{\rm max}-u)^{2\big(\frac{p-1}{p}\big)}\bigg)
\end{align*}
\end{lemma}

\proof
By straightforward calculations, we have
\begin{equation}\label{ff8}
\lim_{u\to u_{\rm max}} \frac{\psi(u)-\psi(u_{\rm max})}{[u_{\rm max}-u]^\frac{p-1}{p}} = -\left(\frac{p}{p-1}\right)^\frac{p-1}{p}2^{-\frac{1}{p}}.    
\end{equation}
For the second term, we compute 
$$\lim_{u\to u_{\rm max}} \frac{\psi(u)-R+\left(\frac{p}{p-1}\right)^\frac{p-1}{p}2^{-\frac{1}{p}}(u_{\rm max}-u)^\frac{p-1}{p}}{(u_{\rm max}-u)^\alpha}=2^{-1-\frac{2}{p}}\frac{1}{2p-1} \left(\frac{p}{p-1}\right)^{2\frac{p-1}{p}}\frac{1}{R}$$
notice that now $u=v_R(t)$, we must have $\alpha=2\frac{p-1}{p}$. The desired expansion follows at once.
\endproof

We are now in a position to introduce the Taylor expansion of $V_R( t,\theta)$ near $ t=0$, which is the following.

\begin{prop}\label{prop stima VR}
Assume that \eqref{ass_esp} holds. Then, as $t\to0^+$ we have
\begin{equation*}
{V_R}_{|(t,\theta)}=2t+\left[\frac{1}{2p-1}k(\theta)+\frac{2(p-1)}{2p-1}\frac{1}{R}\right] t^2+o(t^2)
\end{equation*}
\end{prop}

\proof
It holds 
$$\frac{\partial V_R}{\partial t}\bigg|_{t=0}=\frac{\partial|\dot u_R|^{p-1}}{\partial\rho}\bigg|_{\rho=R} \frac{\partial (\psi\circ u)}{\partial t}\bigg|_{t=0}=-2\,\,\frac{\partial (\psi\circ u)}{\partial t}\bigg|_{t=0}.$$
Now,
\begin{equation}\label{ff7}
\frac{\partial(\psi\circ u)}{\partial t}\bigg|_{t=0}= \lim_{t\to0}\frac{\psi(u_{\rm max})-\psi(u( t,\theta))}{[u_{\rm max}-u( t,\theta)]^\frac{p-1}{p}} \,\frac{[u_{\rm max}-u( t,\theta)]^\frac{p-1}{p}}{ t}=-1   
\end{equation}
due to \eqref{ff8} and \eqref{ass_esp}; so
$$\frac{\partial V_R}{\partial t}\bigg|_{t=0}=2.$$
Regarding the second derivative, we have
\begin{align*}
&\frac{\partial^2 V_R}{\partial t^2}\bigg|_{t=0}=\lim_{t\to0}\frac{\frac{\partial V_R}{\partial t}(t,\theta)-\frac{\partial V_R}{\partial t}(0,\theta)}{t}\\
&=\lim_{t\to0}\frac{\frac{\partial }{\partial\rho} \big|_{\rho=\psi(u)}\big(|\dot u_R|^{p-1}\big)\frac{\partial}{\partial t}(\psi\circ u)-\frac{\partial }{\partial\rho} \big|_{\rho=R}\big(|\dot u_R|^{p-1}\big)\frac{\partial}{\partial t}\big|_{t=0}(\psi\circ u)}{t}\\
&=\lim_{t\to0}\frac{\partial (\psi\circ u)}{\partial t}
\frac{\frac{\partial }{\partial\rho} \big|_{\rho=\psi(u)}\big(|\dot u_R|^{p-1}\big)-\frac{\partial }{\partial\rho} \big|_{\rho=R}\big(|\dot u_R|^{p-1}\big)}{t}+\lim_{t\to0}\frac{\frac{\partial}{\partial t}(\psi\circ u)-\frac{\partial}{\partial t}\big|_{t=0}(\psi\circ u)}{t} \,\,\frac{\partial }{\partial\rho} \bigg|_{\rho=R}|\dot u_R|^{p-1}\\
&=\frac{\partial(\psi\circ u)}{\partial t}\bigg|_{t=0} \,\,\frac{\partial^2 |\dot u_R|^{p-1}}{\partial\rho^2} \bigg|_{\rho=R}\,\,\lim_{u\to u_{\rm max}}\frac{\psi(u)-R}{(u_{\rm max}-u)^\frac{p-1}{p}}\,\lim_{t\to0}\left(\frac{u_{\rm max}-u}{t^\frac{p}{p-1}}\right)^\frac{p-1}{p}\,\,-\,2\frac{\partial^2(\psi\circ u)}{\partial t^2}\bigg|_{t=0}\\
&=(-1)\,\cdot\frac{2}{R}\,\cdot(-1)\,-\,2\,\frac{\partial^2(\psi\circ u)}{\partial t^2}\bigg|_{t=0}
\end{align*}
where in the last equality, we used \eqref{ff7}, \eqref{ff8},  and \eqref{ass_esp}. Moreover, we have
$$\frac{\partial^2(\psi\circ u)}{\partial t^2}= \frac{\partial^2 \psi}{\partial u^2}\left(\frac{\partial u}{\partial t}\right)^2+ \frac{\partial\psi}{\partial u}\,\frac{\partial^2 u}{\partial t^2}.$$
From Lemma \ref{espansione phi} we get the expansion of $\frac{\partial^2 \psi}{\partial u ^2}$ in terms of $(u_{\rm max}-u)$, then we exploit \eqref{ass_esp} to obtain the expansion in terms of powers of $t$ and also to compute $\frac{\partial u}{\partial t}$, we finally get
$$\frac{\partial^2 \psi}{\partial u^2}\left(\frac{\partial u}{\partial t}\right)^2=\frac{1}{p-1}t^{-1}+\frac{1}{(p-1)(2p-1)}\left[\frac{3}{2}k(\theta)+\frac{p-2}{R}\right]+O(t).$$
Arguing in the same way, we obtain
$$\frac{\partial\psi}{\partial u}\frac{\partial^2 u}{\partial t^2}=-\frac{1}{p-1}t^{-1}-2^{-1}\frac{2p^2-p-1}{(p-1)^2(2p-1)}k(\theta)+\frac{1}{(p-1)(2p-1)}\frac{1}{R} +O(t).$$
Therefore 
$$\frac{\partial^2(\psi\circ u)}{\partial t^2}= -\frac{1}{2p-1}k(\theta)+\frac{1}{2p-1}\frac{1}{R}+O(t),$$
and 
$$\frac{\partial^2 V_R}{\partial t^2}\bigg|_{t=0}=\frac{2}{R}-2\left[-\frac{1}{2p-1}k(\theta)+\frac{1}{2p-1}\frac{1}{R}\right]=\frac{2}{2p-1}k(\theta)+\frac{4}{R}\frac{p-1}{2p-1}$$
\endproof

When considering $u$ to be a solution to problem \eqref{sist esterno}, the model solution $u_R$ seen as a function of the distance $t$ is the following
$$v_R=\int_R^{R+t}\left(\rho-\frac{R^2}{\rho}\right)^\frac{1}{p-1}d\rho$$
hence in this case the expansion for $V_R$ is a little different.

\begin{prop}\label{prop stima VR esterna}
Let $(\Om,u)$ a solution to problem \eqref{sist esterno}.
Assume that $u_{\rm max}\le \frac{p-1}{p}|\beta|^p$ and moreover that \eqref{ass_esp} holds. Then, as $t\to0^+$ we have
\begin{equation*}
{V_R}_{|(t,\theta)}=2t+\left[\frac{1}{2p-1}k(\theta)-\frac{2(p-1)}{2p-1}\frac{1}{R}\right] t^2+o(t^2)
\end{equation*}
\end{prop}

%------------------------------------------------------------------

\section{The rigidity results}\label{sez9}

We are now in the position to prove our main rigidity results. 

\proof[Proof of Theorem \ref{teo rigidità a<b e a>b}]
We prove the result in the case $a<u<b$ (the case $b<u<a$ can be proven in the same way). This is the case where $\nabla u$ does not vanish on $\partial\Om$, so it is simpler. In what follows, let $r_{i}$ and $r_o$ be the comparison parameters associated to the solution $(\Omega,u)$. We apply the divergence theorem, since $-\Delta_p u=2$, we get
\begin{align*}
\int_{\Om} \frac{2}{R^2-\Psi^2}d\mu&=-\int_{\Om} \frac{\Delta_p u}{R^2-\Psi^2}d\mu\\
&=\int_{\Om}\langle\,\nabla\left(\frac{1}{R^2-\Psi^2}\right)\, ,\, |\nabla u|^{p-2}\nabla u\,\rangle\,d\mu-\int_{\partial{\Om}}\frac{ |\nabla u|^{p-2}\langle\nabla u,\nu\rangle}{R^2-\Psi^2}\,d\sigma\\
&=\int_{\Om}\frac{2\Psi^\frac{p}{p-1}}{(R^2-\Psi^2)^\frac{2p-1}{p-1}}
|\nabla u|^p d\mu+\int_{\Gamma_{\rm i}}\frac{|\nabla u|^{p-1}}{R^2-\Psi^2}\,d\sigma-\int_{\Gamma_{\rm o}}\frac{|\nabla u|^{p-1}}{R^2-\Psi^2}\,d\sigma
\end{align*}
where $\nu$ denotes the exterior unit normal to $\partial\Omega$. Recalling that $W=|\nabla u|^p$ and $W_R=\left(\frac{R^2-\Psi^2}{\Psi}\right)^\frac{p}{p-1}$ we get
$$\int_{\Om}\frac{2}{R^2-\Psi^2}-\frac{2\Psi^\frac{p}{p-1}}{(R^2-\Psi^2)^\frac{2p-1}{p-1}}|\N u|^p\,d\mu\,=\,\int_{\Omega}\frac{2\Psi^\frac{p}{p-1}}{(R^2-\Psi^2)^\frac{2p-1}{p-1}}(W_R-W)\,d\mu.$$
We observe that $\frac{|\nabla u|^{p-1}}{R^2-\Psi^2}=\frac{1}{\Psi}\left(\frac{W}{W_R}\right)^\frac{p-1}{p}$ and by construction $W=W_R$ on $\Gamma_{\rm i}$ and $\Gamma_{\rm o}$, so 
\begin{equation}\label{ug proof rig non null}
\int\limits_{\Omega}\frac{2\Psi^\frac{p}{p-1}}{(R^2-\Psi^2)^\frac{2p-1}{p-1}}(W_R-W)\,d\mu\,=\,\frac{|\Gamma_{\rm i}|}{r_i}-\frac{|\Gamma_{\rm o}|}{r_o}.  
\end{equation}
Since by Theorem \ref{stime gradiente} it holds $W_R-W\ge0$ in $\Omega$, from the previous identity we obtain that 
$$\frac{|\Sigma_{\rm i}|}{r_i}-\frac{|\Sigma|}{R}\ge0.$$
Now, we observe that by Proposition \ref{prop length bnds Gamma_i e out}, we have 
$$\frac{|\Sigma_{\rm i}|}{r_i}-\frac{|\Sigma|}{R}\le0,$$
hence equality holds. Coming back to \eqref{ug proof rig non null}, this implies that $W\equiv W_R$ in $\Omega$, so using Theorem \ref{stime gradiente} we can conclude that our solution $(\Omega,u)$ is equivalent to $(\Omega_R,u_R)$.
\endproof

Next we deal with the rigidity result for a solution to problem \eqref{sist interno}, or problem \eqref{sist esterno}, namely the case where either $\alpha=0$ or $\beta=0$.

\begin{theorem}\label{teo rigidità caso beta nullo}
Let $(\Omega,u)$ be a solution to problem \eqref{sist interno}, or a solution to problem \eqref{sist esterno}, in this case assume also $M<\frac{p-1}{p}|\beta|^p$. Assume that the asymptotic expansion \eqref{ass_esp} holds for $u$ at $\Sigma$.

Then, $\Om$ is a circular annular domain, and the solution $u$ is rotationally symmetric.
\end{theorem}

\proof
\textbf{Case 1: $(\Omega,u)$ is a solution to problem \eqref{sist interno}.}
Since $\nabla u$ vanishes on the boundary of $\Omega$ (in particular, it vanishes on $\Sigma$), we apply the divergence theorem in $\Om_\delta:=\Om\setminus\{x\in\Om\,|\,u(x)>u_{\rm max}-\delta\}$, and proceeding as in the previous proof, we get
$$\int_{\Om_\delta} \frac{2}{R^2-\Psi^2}d\mu=\int_{\Om_\delta}\frac{2\Psi^\frac{p}{p-1}}{(R^2-\Psi^2)^\frac{2p-1}{p-1}}
|\nabla u|^p d\mu+\int_{\Sigma_{\rm i}}\frac{|\nabla u|^{p-1}}{R^2-\Psi^2}\,d\sigma-\int_{\Sigma_\delta}\frac{|\nabla u|^{p-1}}{R^2-\Psi^2}\,d\sigma$$
where $\nu$ denotes the exterior unit normal to $\partial\Omega_\delta$, and $\Sigma_\delta:=\{x\in\Omega| u(x)=u_{\rm max}-\delta\}$. We get
$$\int_{\Omega_\delta}\frac{2}{R^2-\Psi^2}-\frac{2\Psi^\frac{p}{p-1}}{(R^2-\Psi^2)^\frac{2p-1}{p-1}}|\N u|^p\,d\mu\,=\,\int_{\Omega_\delta}\frac{2\Psi^\frac{p}{p-1}}{(R^2-\Psi^2)^\frac{2p-1}{p-1}}(W_R-W)\,d\mu.$$
hence
$$\int\limits_{\Omega_\delta}\frac{2\Psi^\frac{p}{p-1}}{(R^2-\Psi^2)^\frac{2p-1}{p-1}}(W_R-W)\,d\mu\,=\,\frac{|\Gamma_{\rm i}|}{r_i}-\int_{\Sigma_\delta}\frac{|\nabla u|^{p-1}}{R^2-\Psi^2}\,d\sigma,$$
Moreover, exploiting Proposition \ref{lemma7.1} and Proposition \ref{prop stima VR}, we have
$$\frac{|\nabla u|^{p-1}}{R^2-\Psi^2}=\frac{1}{\Psi}\frac{V}{V_R}\to\frac{1}{\Psi}=\frac{1}{R}, \qquad\text{as }\delta\to0^+,$$
hence in the limit as $\delta\to0^+$
\begin{equation}\label{formula w lungh}
\int\limits_{\Omega}\frac{2\Psi^\frac{p}{p-1}}{(R^2-\Psi^2)^\frac{2p-1}{p-1}}(W_R-W)\,d\mu\,=\,\frac{|\Sigma_{\rm i}|}{r_i}-\frac{|\Sigma|}{R}.
\end{equation}
Since by Theorem \ref{stime gradiente beta=0} it holds $W_R-W\ge0$ in $\Omega$, from \eqref{formula w lungh} we obtain that 
$$\frac{|\Sigma_{\rm i}|}{r_i}-\frac{|\Sigma|}{R}\ge0.$$
Notice that the assumptions of Theorem \ref{stime gradiente beta=0} are in force because the asymptotic expansion \eqref{ass_esp} implies that a 
$\frac{1}{p^2}$-reverse {\L}ojasiewicz inequality holds for $u$
at $\Gamma_{\rm o}=\Sigma$.
Now, we observe that, by Proposition \ref{prop length bnds Gamma_i e out} and Proposition \ref{prop bound derivata norm nulla}, $$\frac{|\Sigma_{\rm i}|}{r_i}-\frac{|\Sigma|}{R}\le0,$$
so that equality holds.
Coming back to \eqref{formula w lungh}, this implies that $W\equiv W_R$ in $\Omega$, so using Theorem \ref{stime gradiente beta=0} we can conclude that our solution $(\Omega,u)$ is equivalent to $(\Omega_R,u_R)$.

\textbf{Case 2: $(\Omega,u)$ is a solution to problem \eqref{sist esterno}.}
We obtain the same rigidity result also in the `external' case.
Notice that the argument employed to prove the result in Case 1 does not work anymore now, since we cannot obtain two opposite inequalities as before; we overcome such difficulty, taking advantage of the previous Pohozaev identity and the isoperimetric inequality. 
By plugging the boundary conditions into the Pohozaev identity \eqref{POHOZAEV}, we get the following identity 
\begin{equation}\label{poho 2}
\int\limits_{\Omega}\left(4u-\frac{2-p}{p}W\right)d\mu=
2\,\frac{p-1}{p}|\beta|^p|A|-4u_{\rm max}|A_\mathrm{int}|.
\end{equation}
Where $\Omega=A\setminus A_{\rm int}$ and $\partial A=\Sigma_{\rm o}$, $\partial A_{\rm int}=\Sigma$.
We proceed by writing the RHS of the previous equality more suitably for our purposes. Let us consider the function $\varphi$
such that 
\begin{equation}\label{ode fi}
\frac{2\varphi-4u}{\dot\varphi-\frac{2-p}{p}}=W_R.
\end{equation}
Solving the ODE \eqref{ode fi}, we get that the function $\varphi$ has the following form
\begin{equation}\label{espressione phi}
\f(u)=2u+\frac{3p-2}{p}\,\frac{1}{\Psi^2-R^2}\,\left(\int_{R}^{\psi(u)}\rho^{-\frac{1}{p-1}}(\rho^2-R^2)^\frac{p}{p-1}\,d\rho+c\right)
\end{equation}
where $k$ is a constant we will fix later.
Using the divergence theorem in the LHS of \eqref{poho 2}, we obtain
\begin{align}\notag
\int\limits_{\Omega}\bigg(4u-&\frac{2-p}{p}W\bigg) d\mu =\int\limits_{\Omega}2\varphi+4u-2\varphi-\frac{2-p}{p}W d\mu\\\notag
&=\int\limits_{\Omega}\dot\f\,W+4u-2\varphi-\frac{2-p}{p}W d\mu-\int\limits_{\partial\Om}\f|\nabla u|^{p-2}\langle\nabla u,\nu\rangle d\sigma\\ \label{rhs poho2}
&=\int\limits_{\Omega}\left(\dot\f-\frac{2-p}{p}\right)\left(W-\frac{2\f-4u}{\dot\f-\frac{2-p}{p}}\right)d\mu+ |\beta|^{p-1}\f_o|\Sigma_{\rm o}|
\end{align}
where by $\f_o$ we denote the value of $\f$ on $\Sigma_{\rm o}$. Finally, from \eqref{poho 2} and \eqref{rhs poho2} and using the fact that $|A|=|\Omega|+|A_{\rm int}|$, we get that
$$\int\limits_{\Omega}\left(\dot\f-\frac{2-p}{p}\right)\left(W-W_R\right)d\mu
=4\left(\frac{p-1}{2p}|\beta|^p-u_{\rm max}\right)|A_{\rm int}|
+2\frac{p-1}{p}|\beta|^p|\Omega|-|\beta|^{p-1}\f_o|\Sigma_{\rm o}|.$$
In what follows, we use the following relations: $|A_{\rm int}|\le \frac{R\,|\Sigma|}{2}$ which comes from the isoperimetric inequality and Proposition \ref{prop bound derivata norm nulla}, 
and $2|\Om|=|\beta|^{p-1}|\Sigma_{\rm o}|$ which is a consequence of the divergence theorem together with the fact that $-\Delta_p u=2$. Recalling that $u_{\rm max}>\left(\frac{p-1}{2p}\right)|\beta|^p$ (see Corollary \ref{cor cond poho}) we obtain
\begin{equation}\label{ineq tri+sro}
\int\limits_{\Omega}\Big(\dot\f-\frac{2-p}{p}\Big)\left(W-W_R\right)d\mu\ge\, T\,|\Sigma|\,+\,S\,|\Sigma_{\rm o}|
\end{equation}
where 
$$T:=\left[\Big(2u_{\rm max}-\frac{p-1}{p}|\beta|^p\Big)R\right]$$
and
$$S:=\left[|\beta|^{p-1}\Big(\frac{p-1}{p}|\beta|^p-\f_o\Big)\right].$$
Now, we observe that if $u$ is a ring-shaped model solution, we have an identity in \eqref{ineq tri+sro} (note that $|A_{\rm int}|= \frac{R\,|\Sigma|}{2}$ in the rigid case), and moreover $W=W_R$, which yields that 
$$T\,R\,+\,S\,r_o=0.$$
Since the values of both $T$ and $S$ do not depend on the choice of a model solution, we can use this relation in the general framework, where we do not know a priori if $u$ is a ring-shaped solution. In particular, we have that
$$\int\limits_{\Omega}\Big(\dot\f-\frac{2-p}{p}\Big)\left(W-W_R\right)d\mu\ge\, T\,|\Sigma|\,+\,S\,|\Sigma_{\rm o}|\,=\, \left(\frac{|\Sigma_{\rm o}|}{r_o}-\frac{|\Sigma|}{R}\right)r_o S.$$
If we are able to find a function $\f$ of the form \eqref{espressione phi} such that 
\begin{itemize}
    \item[i)] $\dot\f-\frac{2-p}{p}>0$
    \item[ii)] $r_o S\ge0$
\end{itemize} 
then we are done, since it holds
$$0\ge \int\limits_{\Omega}\Big(\dot\f-\frac{2-p}{p}\Big)\left(W-W_R\right)d\mu\ge \left(\frac{|\Sigma_{\rm o}|}{r_o}-\frac{|\Sigma|}{R}\right)r_o S \ge 0.$$
This implies that $W=W_R$ in $\Om$ and by Theorem \ref{stime gradiente alfa=0} we get the rotational symmetry of the solution. In order to conclude the proof, we have to verify that there exists $k\in\R$ such that $\f$ satisfies both $i)$ and $ii)$. 
By computing explicitly the derivative of $\f(u)$ we observe that i) holds if and only if
\begin{equation}\label{cond i}
c\ge-\int_{R}^{\psi(u)}\rho^{-\frac{1}{p-1}}(\rho^2-R^2)^\frac{p}{p-1}\,d\rho.
\end{equation}
On the other hand, exploiting the fact that $TR=-Sr_o$, that is 
$$\frac{p-1}{p}|\beta|^p r_o^2-2u_{\rm max}R^2=\frac{3p-2}{p}\int_{R}^{r_o}\rho^{-\frac{1}{p-1}}(\rho^2-R^2)^\frac{p}{p-1}\,d\rho,$$
we have that ii) is verified if and only if
$$c\le \frac{p}{3p-2}R^2\left[-\frac{p-1}{p}|\beta|^p+2u_{\rm max}\right]=:Q.$$
Since $u_{\rm max}>\left(\frac{p-1}{2p}\right)|\beta|^p$, then 
$Q>0.$ A suitable choice for $c$ proves to be $c=0$, which concludes the proof.
\endproof

\begin{obs}\label{no isl}
{\rm 
One can actually prove a sort of No-Island-type result.
More specifically, under the assumptions of Theorem \ref{TEOREMA rigidita}, if $(\Omega,u)$ solves problem \eqref{nostro problema}, then the curve $\Sigma$ divides the domain $\Omega$ into two ring-shaped domains, that is
$$\Omega=\Omega_{\rm i}\sqcup\Sigma\sqcup\Om_{\rm o},$$ 
where $\Omega_{\rm i}$ denotes the inner region of $\Om$ having boundary $\partial\Om_{\rm i}=\Sigma\sqcup\Sigma_{\rm i}$, and $\Om_{\rm o}$ the outer one such that $\partial\Om_{\rm o}=\Sigma\sqcup\Sigma_{\rm o}$. 
The sketch of the proof is the following. Let us assume by contradiction that the smooth closed curve $\Sigma$ defines a region $\Omega'$ which does not contain the inner boundary component $\Sigma_\mathrm{i}$, then we have
$$\begin{cases}
    \Delta_p u=-2\qquad&\text{in }\Omega'\\
    u=u_\mathrm{max} &\text{on }\Sigma.
\end{cases}$$
Let us consider $v_1=u_\mathrm{max}$ and $v_2=u$, clearly $v_1$ is a $p$-harmonic function while $v_2$ is a weak supersolution to the $p$-Laplacian, we can apply a Comparison Principle (see e.g. \cite[Theorem 2.15]{Lindq}) obtaining that $u_\mathrm{max}\le u$ in $\Omega'$, hence it follows that $u\equiv u_\mathrm{max}$ in $\Omega'$, which gives a contradiction. We have thus proven that if $\Sigma$ is a nontrivial closed curve, then it divides the domain $\Omega$ into two annular-type subdomains (the closure of one containing the inner boundary component $\Sigma_\mathrm{i}$ and the closure of the other one containing the outer one $\Sigma_\mathrm{o}$). Finally, self-intersections of the curve $\Sigma$ can be ruled out by invoking a maximum principle.
}
\end{obs}

\proof[Proof of Theorem \ref{TEOREMA rigidita}]
As pointed out in Remark \ref{no isl}, the curve $\Sigma$ divides the domain $\Omega$ into two ring-shaped domains: the inner region $\Omega_{\rm i}$ and the outer one $\Om_{\rm o}$.
Since $|\nabla u|$ is locally constant on $\partial\Om$, without loss of generality, we can assume that $\partial_\nu u=\alpha$ on $\Sigma_{\rm i}$ and $\partial_\nu u=\beta$ on $\Sigma_{\rm o}$, where $\alpha$ and $\beta$ real constants. 

Notice that the asymptotic expansion \eqref{ass_esp}allows us to apply Theorem \ref{teo rigidità caso beta nullo} and Theorem \ref{teo2.1NEW}, in view of Remark \ref{II implica no bnd su p}.

Since $(\Om,u)$ is a solution to problem \eqref{nostro problema}, then $u_{|\Om_{\rm i}}$ solves system \eqref{sist interno} when $M=u_{\rm max}$. We can apply Theorem \ref{teo rigidità caso beta nullo}, and we obtain that $(\Om_{\rm i},u_{|\Om_{\rm i}})$ is rotationally symmetric. We claim that $u_{\rm max}<\frac{p-1}{p}|\nabla u|^p$ on $\Sigma_{\rm o}$. 
Indeed, if $u_{\rm max}\ge\frac{p-1}{p}|\nabla u|^p$ on $\Sigma_{\rm o}$, then from Theorem \ref{teo2.1NEW} we obtain that $\Omega$ must be a disk, which yields a contradiction. Under the condition $u_{\rm max}<\frac{p-1}{p}|\nabla u|^p$ on $\Sigma_{\rm o}$, we can apply Theorem \ref{teo rigidità caso beta nullo}, proving that the solution is radially symmetric. We finally observe that the two portions of radial solutions we have found can be glued, and in particular, up to translations and rescalings, the solution has the form \eqref{model_annular_slus}-\eqref{model annular domain}.
\endproof

%---------------------------------------------------------------------

\begin{ackn}

The authors would like to thank S.Borghini, L. Mazzieri, and A. Pinamonti for useful discussions during the preparation of this work. They are also grateful to E. Pagani and A. Valli for valuable suggestions concerning the derivation of the equation from a laminar flow model and for bringing the useful reference \cite{SerrinBook} to their attention.
The authors are members of the Gruppo Nazionale per l'Analisi Matematica, la Probabilit\`a e le loro Applicazioni (GNAMPA), which is part of the Istituto Nazionale di Alta Matematica (INdAM). They gratefully acknowledge the support by the MUR PRIN-2022JJ8KER grant ``Contemporary perspectives on geometry and gravity''. C.B. has been partially funded by the GNAMPA project CUPE5324001950001 and MUR Department of Excellence project MatMod@TOV.
\end{ackn}

%------------------------------------------------------------------

\bibliographystyle{acm}

\end{document}